\documentclass[12pt,reqno]{amsart}
\usepackage{amsmath,amssymb,amsthm,pinlabel}
\usepackage[usenames]{color}
\usepackage[draft]{hyperref}

\usepackage[all]{xy}

\newcommand{\BZ}{\mathbb{Z}} 
\newcommand{\BN}{\mathbb{N}} 
\newcommand{\BR}{\mathbb{R}}

\newcommand{\CA}{\mathcal{A}}
\newcommand{\CB}{\mathcal{B}}
\newcommand{\CC}{\mathcal{C}}

\newcommand{\CL}{\mathcal{L}}

\newcommand{\CQ}{\mathcal{Q}}

\newcommand{\CT}{\mathcal{T}}

\newcommand{\co}{\colon\thinspace}

\newcommand{\la}{\langle}
\newcommand{\ra}{\rangle}
\newcommand{\bd}{\partial}
\newcommand{\inv}{^{-1}}

\theoremstyle{plain}
\newtheorem{theorem}{Theorem}
\newtheorem*{GeoTwist}{Theorem \ref{th:geo-twist}}
\newtheorem*{AlgTwist}{Theorem \ref{th:alg-twist}}

\newtheorem{lemma}[theorem]{Lemma}
\newtheorem{proposition}[theorem]{Proposition}
\newtheorem{corollary}[theorem]{Corollary}

\newtheorem*{claim*}{Claim}

\theoremstyle{definition}
\newtheorem{definition}[theorem]{Definition}

\newtheorem{example}[theorem]{Example}
\newtheorem{remark}[theorem]{Remark}

\numberwithin{theorem}{section}

\let\ssection=\section
\renewcommand{\section}{\setcounter{equation}{0}\ssection}

\DeclareMathOperator{\Out}{Out}

\DeclareMathOperator{\vol}{vol}

\DeclareMathOperator{\Lip}{Lip}

\begin{document}


\title[Relative twisting in Outer space]{Relative twisting in Outer space}

\author[M.~Clay]{Matt Clay}
\address{Dept.\ of Mathematics \\
  Allegheny College\\
  Meadville, PA 16335}
\email{\href{mailto:mclay@allegheny.edu}{mclay@allegheny.edu}}

\author[A.~Pettet]{Alexandra Pettet}
\address{Dept.\ of Mathematics \\
  University of British Columbia\\
  Vancouver, BC V6T 1Z2}
\email{\href{mailto:alexandra@math.ubc.ca}{alexandra@math.ubc.ca}}

\thanks{\tiny The first author is partially supported by NSF grant
  DMS-1006898. The second author was partially supported by NSF grant
  DMS-0856143, NSF RTG DMS-0602191, EPSRC grant EP/D073626/2, and 
  an NSERC Discovery Grant.}

\begin{abstract}
  Subsurface projection is indispensable to studying the
  geometry of the mapping class group and the curve complex of a
  surface.  When the subsurface is an annulus, this projection is
  sometimes called \emph{relative twisting}.  We give two alternate
  versions of relative twisting for the outer automorphism group of a
  free group. We use this to describe sufficient conditions for when a
  folding path enters the {\it thin} part of Culler-Vogtmann's Outer
  space.  As an application of our condition, we produce a sequence of
 fully irreducible outer automorphisms whose axes in Outer space travel through graphs with
  arbitrarily short cycles; we also describe the asymptotic behavior of their translation lengths.
\end{abstract}
\maketitle


\section{Introduction}\label{sc:intro}

Culler and Vogtmann gave the first account of {\it Outer space $CV_k$}
in their 1986 paper \cite{Culler-Vogtmann}: elements are finite marked
projectivized metric graphs with fundamental group $F_k$, the rank $k$
non-abelian free group, and two graphs are close when the lengths of
some finite collection of elements of $F_k$ are close.  By considering
the universal covers of the marked graphs, $CV_k$ is also described as
the space of free simplicial minimal isometric 
actions of $F_k$ on $\BR$--trees.
Topologically, $CV_k$ has the structure of a contractible simplicial
complex (missing some faces) on which $\Out F_k$ acts properly and
simplicially by changing markings. Metrically, however, Outer space
remains largely a mystery. Much of the conjectural picture for Outer
space geometry comes from Teichm\"uller theory, where the
Teichm\"uller metric, the Weil-Petersson metric, and the Thurston
metric have been defined and extensively studied. Unfortunately Outer
space lacks much of the structure that paves the way for these
metrics; perhaps most notably, $CV_k$ is {\it not} a manifold.

Of the three metrics on Teichm\"uller space mentioned above, only the
third, the Thurston metric, has been interpreted in the Outer space
setting; there it is more commonly referred to as the {\it Lipschitz
  metric}. Features of this metric were recorded by
Francaviglia-Martino in \cite{un:FM1,un:FM2}. 

Algom-Kfir (see also Hamenst\"adt \cite{un:Hamenstadt}) proved that
axes of fully irreducible elements of $\Out F_k$ are {\it strongly
  contracting}, so that $CV_k$ exhibits a characteristic of negative
curvature in these directions. Her result was anticipated by a theorem
of Minsky \cite{ar:Minsky}, 
which showed that Teichm\"uller geodesics contained in the
$\epsilon$-{\it thick} part of Teichm\"uller space are strongly
contracting, uniformly depending on $\epsilon$. Algom-Kfir's
contraction constants depend on the outer automorphisms to which they
belong.  The question of whether these constants only depend on the
geometry of the graphs along the axes has not been addressed.

For $\epsilon > 0$, we define $CV_k^\epsilon$ as the subset of $CV_k$
consisting of graphs that contain a cycle of length less than
$\epsilon$.  We should perhaps resist calling $CV_k^\epsilon$ the
``thin part'' of Outer space as it is not clear that Algom-Kfir's
theorem extends uniformly to geodesics in the complement of
$CV_k^\epsilon$.  Nevertheless, this set does hold some nice
properties analogous to those of the thin part of Teichm\"uller space;
for instance, the cusps of the quotient $CV_k/ \Out F_k$ are contained
in $CV_k^\epsilon$, and the quotient $(CV_k -CV_k^\epsilon)/ \Out F_k$
of the complement is quasi-isometric to $\Out F_k$.

The main results of this paper provide conditions, akin to those of
Rafi \cite{ar:Rafi05} in the setting of Teichm\"uller space, that
guarantee that a geodesic or an axis of a fully irreducible element
travel through $CV_k^\epsilon$.  Our criteria are based on a notion of
{\it relative twisting} in Outer space. We come at this from two
different points of view, each motivated by the quest to find
satisfactory analogues of subsurface projection and relative twisting
from the theory of mapping class groups \cite{ar:MM00,ar:FLM01}.

\medskip \noindent {\bf Geometric:} Our first approach to relative twisting directly
adapts the original geometric definition to free
groups.  We give a pairing $\tau_a(G,G')$ between two graphs $G,G' \in
CV_k$ relative to some nontrivial $a \in F_k$, which we define by means
of the \emph{Guirardel core} \cite{ar:Gu05}.  This is a certain
2--complex associated to the graphs that provides a means of selecting a geometry
for $F_k$ that ``sees'' both $G$ and $G'$.  We obtain a condition on
the graphs that, when satisfied, enables us to construct a connecting
geodesic between them, traveling through $CV_k^\epsilon$.

\begin{GeoTwist}
  Suppose $G,G' \in CV_k$ with $d = d_L(G,G')$ such that $\tau_a(G,G')
  \geq n+2$ for some $a \in F_k$.  Then there is a geodesic $\alpha\co
  [0,d] \to CV_k$ such that $\alpha(0) = G$ and $\alpha(d) = G'$ and
  for some $t \in [0,d]$, we have $\ell_{\alpha(t)}(a) \leq 1/n$.  In
  other words, $\alpha([0,d]) \cap CV_k^{1/n} \neq \emptyset$.
\end{GeoTwist}

As a corollary, we get the following lower bound the distance between
two marked graphs in $CV_k$.  

\begin{corollary}
  Suppose $G,G' \in CV_k$ and $G'$ does not have a cycle of length
  less than $\epsilon$.  Then:
  \[ d_L(G,G') \geq \log \sup_{1 \neq a \in F_k} \epsilon\tau_a(G,G')\]
\end{corollary}

\begin{proof}
  Let $a \in F_k$ be nontrivial.  If $\tau_a(G,G') \geq n$, then by
  Theorem \ref{th:geo-twist}, there is a geodesic $\alpha\co [0,d] \to
  CV_k$ such that $\alpha(0) = G$ and $\alpha(d) = G'$, and for some $t
  \in [0,d]$ that $\ell_{\alpha(t)}(a) \leq 1/n$.  As $G'$ does not have a cycle of length less than
  $\epsilon$, it will follow from Proposition \ref{prop:loops} that $d_L(\alpha(t),G') \geq
  \frac{\epsilon}{1/n} = \epsilon n$.  As $\alpha(t)$ is on a geodesic
  from $G$ to $G'$, the corollary holds.
\end{proof}

The similar lower bound for Teichm\"uller space is a special case of a
theorem of Rafi \cite{ar:Rafi07}.

\medskip \noindent {\bf Algebraic:} 
The second point of view to relative twisting gives a
pairing $\tau_a(T,\Lambda)$ between a tree $T \in \overline{CV}_k$ and
an algebraic lamination $\Lambda$ of $F_k$ relative to some nontrivial
$a \in F_k$.  This pairing measures how the axes of $a$ in $T$ overlap with
the leaves of the lamination.  It is similar to the notion of
``twisting'' used by Alibegovi\'c \cite{ar:Al02}.  We obtain a criterion
that implies that the axis of a fully irreducible element travels through
$CV_k^\epsilon$ in terms of its unstable tree and lamination.

\begin{AlgTwist}
  Suppose $\phi \in \Out F_k$ is fully irreducible, with unstable tree $T_-$ and
  lamination $\Lambda_-$ such that $\tau_a(T_-,\Lambda_-) \geq n+4$
  for some $a \in F_k$.  Then given any train-track $G$, there is an
  axis $\CL_\phi$ for $\phi$ that contains $G$ and a graph $G_0$ with 
  $\ell_{G_0}(a) \leq 1/n$. In other words, $\CL_\phi
  \cap CV_k^{1/n} \neq \emptyset$.
\end{AlgTwist}

As an application of Theorem \ref{th:alg-twist}, we examine outer
automorphisms of $F_k$ that are products of powers of two Dehn twists
$\delta_1$ and $\delta_2$ which ``fill'' in an appropriate sense.  We
show (Section \ref{sc:example}) that axes for
$\delta_1^n\delta_2^{-n}$ travel through graphs with a cycle of length
$\sim 1/n$.  Moreover, we can estimate their translation lengths on
$CV_k$; we compute that they grow logarithmically in $n$ (Theorem
\ref{th:twist-translation}).

The proofs of Theorems \ref{th:geo-twist} and \ref{th:alg-twist} are
similar.  In both cases we show that large relative twist implies the
existence of a certain path that contains a large power of $a$
(Propositions \ref{prop:ntwist-vp} and \ref{prop:ncover-vp}).  These
paths, called vanishing paths, are folded, either in the map $G \to
G'$ or in a train-track map representing $\phi$, and are homotopically
trivial in the image.  The most efficient way to fold over a loop
representing $a$ several times is to first make $a$ a short loop
(Proposition \ref{prop:vpimpliessmallcycle}).

It appears likely that our definition of algebraic twist (at least as
used in Theorem \ref{th:alg-twist}) is a special case of our
definition of geometric twist.  We anticipate investigating this
relationship in a further paper.

The paper is organized as follows. In Section \ref{sc:prelim} we
review some of the basic theory of Outer space and the Lipschitz
metric, irreducible outer automorphisms, train-track maps, and
laminations; only Section \ref{ssc:nvp} contains some new
material. As this section is already lengthy, some background, such as a summary of currents for free groups, is suppressed until it is needed in Section \ref{sc:example}. In Section \ref{sc:gtwist}, following an outline of some properties of Guirardel's core
and a brief review of relative twisting for the mapping class group, the
first, ``geometric,'' analogue of relative twisting for $\Out F_k$ is
given. Section \ref{sc:atwist} is concerned
with the second, ``algebraic,'' notion of relative twisting. Each of
Sections \ref{sc:gtwist} and \ref{sc:atwist} conclude with a
proposition essential to the proofs of the main theorems, found in
Section \ref{sc:smallcycle}. In Section \ref{sc:example}, we bring
together results from Section \ref{sc:smallcycle} and previous papers
of the authors \cite{un:CP2,ar:CP} to describe a method for
constructing geodesic axes of fully irreducible elements which enter
the thin part of Outer space.

\medskip
\noindent {\bf Acknowledgments} The authors would like to thank Kasra
Rafi and Juan Souto for helpful conversations concerning this project.
Additionally, the authors thank the referee for a careful reading of this work with helpful suggestions.


\section{Preliminaries}\label{sc:prelim}


\subsection{Outer space}\label{ssc:cv}

We begin by fixing a generating set of the free group $F_k = \langle
x_1, \ldots, x_k \rangle$. Let $G$ be a simplicial graph, i.e., a
one-dimensional cell complex, with $\pi_1(G)$ isomorphic to
$F_k$. Let $R$ be a wedge of $k$ (oriented) circles, with each circle
identified to one of the generators of $F_k$. Then by a {\it marking}
of $G$ we mean a homotopy equivalence $\rho\co R \to G$. From the map
$\rho_\ast\co \pi_1(R) \to \pi_1(G)$, we then have an identification
of $F_k$ with $\pi_1(G)$. Given two marked metric graphs $\rho_1\co R
\to G_1$ and $\rho_2\co R \to G_2$, a map $f: G_1 \to G_2$ is a {\it
  change of marking} if it is linear on edges, and if $f \circ
\rho_1\co R \to G_2$ is homotopic to $\rho_2\co R \to G_2$. A {\it
  topological representative} of $\phi \in \Out F_k$ is a marked graph
$\rho \co R \to G$, together with a self homotopy equivalence $g\co G
\to G$, so that the homotopy equivalence $\rho\inv \circ g \circ \rho
\co R \to R$ induces $\phi$ on $\pi_1(R) = F_k$.

We denote by $cv_k$ the {\it unprojectivized (Culler--Vogtmann) Outer
  space} consisting of marked metric graphs $G$, where $\pi_1(G) =
F_k$ and the degree of every vertex of $G$ is at least 3. Two points
$\rho_1\co R \to G_1$ and $\rho_2\co R \to G_2$ in $cv_k$ are
equivalent if there is an isometry $\iota\co G_1 \to G_2$ so that
$\iota \circ \rho_1$ is homotopic to $\rho_2$. An alternate
description of $cv_k$ is as the space of free minimal isometric 
$F_k$--actions
on simplicial trees, and we will alternate freely between treating
Outer space as a space of trees and as a space of graphs.  There is a
right action of $\Out F_k$ given by precomposing the marking (or $F_k$--action) by a representative of the outer automorphism.  Outer space is defined as the
projectivization of $cv_k$: $CV_k = cv_k/\BR_{>0}$; it can be
identified with the subspace of $cv_k$ consisting of marked graphs whose
edge lengths sum to 1.

To simplify the notation for elements in Outer space, we denote a
marked metric graph $\rho\co R \to G$ simply by $G$. A {\it path} in
$G$ is a continuous map $\alpha\co I \to G$, where $I$ is an interval
of $\BR$. For convenience, and when it is clear from context, $\alpha$
may denote either the map or its image in $G$; while $[\alpha]$ will
denote the image of $\alpha$ after ``pulling it tight,'' 
i.e., the
image of any immersed homotopy (relative endpoints) representative of
$\alpha$.  We then use $L_G(\alpha)$ to denote the length of
$[\alpha]$ in $G$.  For an element $a \in F_k$, we write $\ell_G(a)$
to denote the minimal length of a loop in $G$ representing the
conjugacy class of $a$.  

For points
$x,y \in T$, we use $[x,y]$ to denote the image of the unique tight
path connecting $x$ and $y$ in $T$.  For an $F_k$--tree $T$ and an
element $a \in F_k$, we write $\ell_T(a)$ to denote the minimal
translation length of $a$ in $T$.  If $\ell_T(a) \neq 0$, then $a$ has
an invariant axis $T^{\la a \ra}$ and $\vol(T^{\la a \ra}/\la a \ra) =
\ell_T(a)$.  If $G$ is a graph with fundamental group $F_k$, then
$\widetilde{G}$ denotes the universal cover of $G$, with a chosen base
point so that there is an $F_k$--action on $\widetilde{G}$.  Clearly
$\ell_G = \ell_{\widetilde G}$.

Using the description as a space of tree actions, $cv_k$ is
topologized via the the \emph{axes topology}.  That is, a tree $T \in
cv_k$ is identified with a point in $\BR^{F_k}$ 
by the coordinates
$(\ell_T(g))_{g \in F_k}$ \cite{ar:CM87}.  Cohen and Lustig proved
that the space of very small actions on $\BR$--trees contains the
closure $\overline{cv}_k$ of $cv_k$ \cite{ar:CL95}.  The converse,
that every very small minimal action on an $\BR$--tree is the limit of
free minimal simplicial actions, was shown by Bestvina and Feighn
\cite{un:BF}. 
Recall that an action of $F_k$ on an $\BR$--tree is {\it
  very small} if arc stabilizers are trivial or maximal cyclic, and
the stabilizer of any tripod is trivial.  

Given two points $G_1$ and $G_2$ in the projectivized Outer space
$CV_k$, let $f\co G_1 \to G_2$ be a change of marking, and denote by
$\sigma(f)$ the maximal slope of $f$ (recall that $f$ is linear on
edges). We have the following proposition, due to White (see
\cite{un:Algom-Kfir,un:Bestvina}):
\begin{proposition}[White]\label{prop:loops}
  Let $G_1,G_2$ be two graphs in $CV_k$.  Then:
  \begin{equation*}
    \inf\{\sigma(f) \ | \ f\co G_1 \to G_2 \ \mbox{\rm change of marking} \} = 
    \sup_{1 \neq a \in F_k} \frac{\ell_{\widetilde{G}_2}(a)}{\ell_{\widetilde{G}_1}(a)}
  \end{equation*}
  Moreover both inf and sup are realized.
\end{proposition}

For $G_1$ and $G_2$ in $\Out F_k$, let $\sigma(G_1,G_2)$ be the value
in Proposition \ref{prop:loops}. We define a function $d_L: CV_k
\times CV_k \to \BR_{\geq 0}$ by
\[
d_L(G_1,G_2) = \log \sigma(G_1,G_2).
\]
Its only failure to be a distance is that it is not symmetric; it is
not hard to construct examples of $G_1, G_2 \in CV_k$ with
$d_L(G_1,G_2) \neq d_L(G_2,G_1)$ (see \cite{un:A-KB}).  In spite of
this anomaly, we will refer to $d_L$ as the {\it Lipschitz metric}
on $CV_k$.  We remark that it is known that the minimal Lipschitz
constant, taken over all continuous maps $f\co G_1 \to G_2$ such that $f
\circ \rho_1$ is homotopic to $\rho_2$, is achieved by a map that is
linear on edges (\cite{un:Algom-Kfir, un:FM2}).

\begin{example}\label{ex:twist}
  We present an example of computing distances in $CV_k$ that will be
  relevant to those examples constructed in Section \ref{sc:example}.  Fix
  a basis $\CT = \CA \cup \{ t \}$ of $F_k$ and an element $c \in \la \CA \ra$
  that is cyclically reduced with respect to $\CT$.  Consider the
  Cayley tree $T$ and the marked graph $G = T/F_k$, metrized so that
  all edge lengths are equal to $1/k$.

  Let $\delta$ be the automorphism that sends $t$ to $ct$ and acts as
  the identity on $\la \CA \ra$.  Then there is a change of
  marking map $f \co G \to G\delta^n$ defined by subdividing the edge
  corresponding to $t$ into $n+1$ edges and sending each of the first
  $n$ edges over the edge path for $c$ and the last edge over the edge
  corresponding to $t$.  Therefore, the image of the edge $t$ has
  length:
  \[ n\ell_{G}(c) + \frac{1}{k} = \frac{nk\ell_G(c)
    + 1}{k} \] and hence the edge $t$ has been stretched by
  $nk\ell_G(c) + 1$.  Since the edge corresponding to $t$ is the
  only edge stretched and since it is mapped to a tight loop 
  we have that:
  \[ d_L(G,G\delta^n) = \log (nk\ell_G(c) + 1). \] In the
  terminology from the proof of Theorem 2.1 in \cite{un:Algom-Kfir},
  the loop $t$ is the subgraph $G_f \subset G$ and it is a legal loop.
  
  The automorphism $\delta$ is an example of a {\it Dehn twist automorphism}
  (see Section \ref{sc:example}); in this case corresponding to the Bass--Serre
  tree arising from the HNN-extension $\la \CA,c_0, t \,  | \, t\inv ct = c_0
  \ra$.  We refer to such a tree as a \emph{cyclic tree}.  For the case of
  a cyclic tree dual to an amalgamated free product $\la \CA \ra
  *_{\la c \ra} \la c,\CB \ra$ and its associated Dehn twist ($a
  \mapsto a$, $b \mapsto cbc\inv$), one can also show, using similar
  methods, that the distance from $G$ to $G\delta^n$ is approximately
  $\log n$.  In this case though, the obvious map sending edges
  corresponding to elements $b \in \CB$ to $c^nbc^{-n}$ is not the
  optimal map.  Instead one sends $b$ to $c^{n/2}bc^{-n/2}$ and edges
  corresponding to elements $a \in \CA$ to $c^{-n/2}ac^{n/2}$.

  We remark for use in Section \ref{sc:example}, that
  $d_L(G\delta^n,G) = d_L(G,G\delta^n)$.  Indeed, $d_L(G\delta^n,G) =
  d_L(G,G\delta^{-n})$ and the same argument as above shows that
  $d_L(G,G\delta^{-n}) = \log(nk\ell_G(c^{-1}) + 1)$.  But of course
  $\ell_G(c) = \ell_G(c^{-1})$.
\end{example}


\subsection{Bounded backtracking}\label{ssc:bbt}

Suppose that $f \co T \to T'$ is a continuous map, where $T$ and $T'$
are trees. We say that $f$ has \emph{bounded backtracking} if there is
a constant $C$ such that for any path $[x,y] \subset T$ from $x$ to
$y$ in $T$, and any $z \in [x,y]$, necessarily
$d_{T'}([f(x),f(y)],f(z)) \leq C$. 
We denote by $BBT(f)$ the minimal
such constant $C$. We note that for any given $T \in cv_k$ and $T' \in
\overline{cv}_k$, that any $F_k$--equivariant map $f \co T \to T'$ has
bounded backtracking. Moreover $BBT(f) \leq \Lip(f) \vol_T(T/F_k)$,
where $\Lip(f)$ is the Lipschitz constant of the map $f$
\cite{ar:BFH97}. In particular, if $T$ and $T' $ are contained in the
projectivized space $CV_k$, then $BBT(f) \leq \Lip(f)$.

For a path $\alpha \subset T$, denote by $\alpha\dagger_L$ the path
obtained by deleting the extremal paths of length $L$. The following
is an easy consequence of bounded backtracking.

\begin{lemma}\label{lm:surviving-subsegment}
  Suppose that $T,T' \in \overline{cv}_k$, that $f\co T \to T'$ is an
  $F_k$--equivariant map that has bounded backtracking, and that
  $\ell\co \BR \to T$ is a parametrized geodesic.  If $L > BBT(f)$,
  and if for some interval $I \subset \BR$, a tight path $\alpha
  \subset T'$ is contained in $[f(\ell(I))]\dagger_L$, then
  necessarily $\alpha \subset [f(\ell(I'))]$ for any interval $I'
  \supset I$.
\end{lemma}

\begin{proof}
  Let $I = [x,y] \subset \BR$ and assume that the hypotheses of the lemma hold, so that $d_{T'}(f(\ell(x)),\alpha) \geq L$ and $d_{T'}(f(\ell(y)),\alpha) \geq L$.  Next let $x' \in I$ be such
  that $f(\ell(x'))$ is the endpoint of $\alpha$ closest to
  $f(\ell(x))$, and let $y' \in I$ be such that $f(\ell(y'))$ is the
  endpoint of $\alpha$ closest to $f(\ell(y))$.

  Now suppose that $\alpha \not\subset [f(\ell(I'))]$ for some
  interval $I'$ that contains $I$.  As $T'$ is a tree, either
  there exists an $x'' \in I' - I$ such that $f(\ell(x'')) =
  f(\ell(x'))$ or there exists an $y'' \in I' - I$ such that
  $f(\ell(y'')) = f(\ell(y'))$; without loss of generality, we
  assume the former. Then the path $[\ell(x'),\ell(x'')]
  \subset T$ and the point $\ell(x) \in [\ell(x'),\ell(x'')]$ violate
  bounded backtracking, as $$d_{T'}([f(\ell(x'),f(\ell(x'')],f(\ell(x)))
  \geq L > BBT(f).$$
\end{proof}


\subsection{Irreducible elements and train-tracks}\label{ssc:iwip-tt}

Let $G$ be a graph. A {\it turn} is an unordered pair of oriented edges that share a common initial vertex. Letting $\bar{e}$ denote the edge $e$ with opposite orientation, we say that an edge path
 $\alpha$ {\it crosses} a turn $\{ e_1, e_2 \}$ if it contains an occurrence of either $\bar{e}_1 {e}_2$ or $\bar{e}_2 e_1$.

Let $g\co G \to G$ be a homotopy equivalence that is linear on edges, mapping edges to edge paths. Then $g$ induces a map on the set of turns of $G$, as follows. Let $v$ be the common initial vertex of the edges of a turn $\{ e_1,e_2 \}$. Some initial segment of $e_1$ is mapped onto an edge $e_1'$ based at $g(v)$, while some initial segment of $e_2$ onto an edge $e_2'$, also based at $g(v)$. We assign 
$g(\{ e_1, e_2 \}) = \{ e_1',e_2' \}$.

The homotopy equivalence $g\co G \to G$
is a {\it train-track map} if there is a collection
$\CL\CT$ of turns such that: 
\begin{enumerate}

\item $\CL\CT$ is closed under iteration of $g$ and

\item for an edge $e \subset G$, any turn crossed by $g(e)$
  is in $\CL\CT$.

\end{enumerate}
The unordered pairs of $\CL\CT$ are called \emph{legal turns}, while
an unordered pair of turns not in $\CL\CT$
is called an \emph{illegal turn}.  A path is \emph{legal} if it only
crosses legal turns. We will regularly refer to the underlying graph $G$ 
as a ``train-track.''

An element of $\phi \in \Out F_k$ is {\it reducible} if some conjugacy
class of a proper free factor of $F_k$ is $\phi$--periodic; otherwise
$\phi$ is {\it irreducible}.  Bestvina and Handel proved that every
irreducible element of $\Out F_k$ has a topological representative
that is a train-track map \cite{ar:BH92}. If $g\co G \to G$ 
is a train-track map representing an
irreducible element of $\Out F_k$, there is a metric on $G$ such that
$g$ linearly expands each edge of $G$ by the same factor $\lambda$,
called the \emph{expansion factor}. This factor is the
Perron--Frobenius eigenvalue of the transition matrix for $g$; a
positive eigenvector for this eigenvalue specifies the metric on $G$.

All proper powers of a {\it fully irreducible} element $\phi$ of $\Out
F_k$ are irreducible. A fully irreducible element $\phi$ has the
property that its minimal displacement in the Lipschitz metric is
related to its expansion factor $\lambda_\phi$ by
\[
\min_{G \in CV_k}d_L(G,G\phi) = \log(\lambda_\phi).
\]
Moreover, this minimum is realized by a train-track map for $\phi$.
This relationship is one reason we choose not to symmetrize the
Lipschitz metric, for typically the expansion factor of a fully
irreducible element is not equal to that of its inverse.  See for
example \cite{ar:HM07}. 

For a fully irreducible element $\phi \in \Out F_k$ and any tree
$T \in CV_k$, the sequence $T\phi^n$ has a well-defined limit 
$T_+(\phi)$ in
$\overline{CV}_k$, called the \emph{stable tree} of $\phi$ \cite{ar:BFH97,ar:LL03}.  
The {\it unstable tree}  for
$\phi$, denoted by $T_-(\phi)$, is the stable tree for
$\phi\inv$, i.e., $T_-(\phi) = T_+(\phi\inv)$.  For an explicit
description see \cite{un:HM}.  We further note that if $T \in
\overline{CV}_k - \{ T_-(\phi)\}$, then $T\phi^n$ converges to
$T_+(\phi)$ \cite{ar:LL03}.

\subsection{Geodesics in Outer space}\label{ssc:axes}

Next, for an interval $I \subset \BR$, we describe paths $I \to CV_k$ known as {\it folding lines}. We will be concerned with two types
of such paths: those which connect two points $G_1$ and $G_2$ in the
interior $CV_k$, and those which are axes of fully irreducible elements. The latter
were studied by Algom-Kfir in \cite{un:Algom-Kfir}.

For $G_1, G_2 \in CV_k$, let $f\co G_1 \to G_2$ be a change of marking
map whose Lipschitz constant realizes $\sigma(G_1,G_2)$. Find a path
$\tilde{\alpha}_1$ based at $G_1$ contained in an open simplex of
(unprojectivized) $cv_k$, along which edges of $G_1$ shrink just until
the map induced by $f$ stretches every edge of the resulting graph by
$\sigma(G_1,G_2)$; note that the lengths of those edges of $G_1$ that
are stretched by exactly $\sigma(G_1,G_2)$ do not change along
$\tilde{\alpha}_1$. Let the endpoint of the corresponding path
$\alpha_1$ in (projectivized) $CV_k$ be $H_1$, with the change of
marking map $h\co H_1 \to G_2$ induced by $f$. We choose a
parameterization $\alpha_1\co [0,d_L(G_1,H_1)] \to CV_k$ by arclength.

 Now we construct a path $\alpha_2: [0,d_L(H_1,G_2)] \to CV_k$ with
 $\alpha_2(0) = H_1$ and $\alpha_2(d_L(H_1,G_2)) = G_2$. First
 subdivide the edges of $H_1$ to obtain a graph $H_1'$ so that the
 preimage of vertices in $G_2$ consists of vertices in $H_1'$, while the
 induced map $h \co H_1' \to G_2$ remains cellular. Select a vertex
 $v$ of $H_1'$ at which two edges $e_1$ and $e_2$ identified by $h$
 are based. Let $H(t)$ be the graph obtained from $H_1'$ by folding
 the initial segments of length $(1-e^{-t})$ of $e_1$ and $e_2$; let
 $\alpha_2(t)$ be the graph in $CV_k$ obtained from $H(t)$ by
 ``forgetting'' valence 2 vertices. Note that $d_L(H_1,\alpha_2(t)) =
 t$. Define $\alpha_2(t)$ in this way until $e_1$ and $e_2$ are
 completely identified; then repeat the above with the resulting
 graph. Continue this process until the map induced by $f$ is an
 immersion in $G_2$; this is a finite process as $H_1'$ has a finite
 number of vertices. Note that the immersion is necessarily an
 isometry as every edge is stretched by the same factor; thus the fold
 line just constructed connects $H_1$ to $G_2$ in $CV_k$. Finally,
 let $\alpha$ be the concatenation of $\alpha_1$ and $\alpha_2$; a
 path based at $G_1$ and terminating at $G_2$. Francaviglia and
 Martino \cite[Theorem 5.5]{un:FM2} proved that $\alpha \co
 [0,d_L(G_1,G_2)] \to CV_k$ is a geodesic.

 We can now describe a geodesic axis for a fully irreducible element
 $\phi$ of $\Out F_k$. Let $\lambda_\phi$ be the expansion factor of
 $\phi$, and let $G$ be a train-track. To obtain a parametrized
 geodesic axis
 for $\phi$, first find a folding path $\alpha: [0,\log(\lambda_\phi)]
 \to CV_k$ as above, connecting $G$ to $G\phi$. Then define the graph
 $\alpha(t) = \alpha(t-n(t) \log(\lambda_\phi)) \phi^{n(t)}$, where $n(t)$
 is the integer $\lfloor \frac{t}{\log(\lambda_\phi)} \rfloor$, and
 let $\mathcal{L}_\phi$ denote the image of $\alpha$. Algom-Kfir
 \cite[Proposition 3.5]{un:Algom-Kfir} showed that 
 $\alpha: \BR \to CV_k$ 
 is a geodesic parametrized by arclength.


\subsection{Nielsen and vanishing paths}\label{ssc:nvp}

Suppose $G$ is a graph with a homotopy equivalence $g \co G \to G$. We
make note of two special types of paths in $G$ and collect some relevant results that will 
be useful for us in the sequel.

First, a path $\alpha \subset G$ is a \emph{Nielsen path} if $[g(\alpha)] =
[\alpha]$; it is \emph{indivisible} if it is not a concatenation of nontrivial Nielsen paths, so that  
any Nielsen path is a concatenation of indivisible Nielsen paths. 

\begin{theorem}[\cite{un:HM}, Corollary 2.14]\label{th:Np-or-vp}
  Suppose that $\phi \in \Out F_k$ is fully irreducible with stable tree $T_+ = T_+(\phi)$, and
  that $g \co G \to G$ is a train-track representative.  Then there is a
  surjective $F_k$--equivariant map $f_g \co \widetilde{G} \to T_+$
  such that if $[x,y] \subset \widetilde{G}$ is a lift of a path
  $\alpha \subset G$ and $f_g(x) = f_g(y)$, then for some $m \geq 0$,
  the path $[g^m(\alpha)]$ is either a Nielsen path or trivial.
\end{theorem}

\begin{definition}\label{def:ncover}
  Let $a \in F_k$ and $T \in \overline{cv}_k$ be such that $\ell_T(a)
  \neq 0$. We say that a path $\alpha \subset T$ (possibly infinite)
  \emph{n--covers} $a$ if $L_T(\alpha \cap T^{\la a \ra}) \geq
  n\ell_T(a)$.  In other words, $\alpha$ overlaps with a segment of
  the axis of $a$ for length at least $n\ell_T(a)$; there is a
  point $x \in \alpha$ such that $a^nx \in \alpha$ as well.

  Similarly, given $G \in CV_k$, we say a path $\alpha \subset G$
  \emph{n--covers} $a$ if a lift of $\alpha$ to $\widetilde{G}$
  $n$--covers $a$. In other words, $\alpha$ decomposes into $\alpha =
  \beta \cdot \alpha_0 \cdot \beta'$, where $\alpha_0$ is the loop
  representing the conjugacy class of $a^n$.
\end{definition}

\begin{lemma}\label{lem:onlyiNP}
  Let $\phi$ and $g\co G \to G$ be as in Theorem
  \ref{th:Np-or-vp}. Suppose that a Nielsen path $\alpha$ in $G$
  $n$--covers $a$, for some $a \in F_k$ with $\ell_{T_+}(a) > 0$ and
  $n \geq 2$. Then there is a subpath $\alpha_0$ of $\alpha$ which
  $n$--covers $a$ and which is contained in $\alpha' \dagger_\epsilon$ for some indivisible Nielsen path $\alpha'$ and $\epsilon > 0$.
\end{lemma}
\begin{proof}
  Suppose that $\alpha_0$ is a shortest subpath of $\alpha$ that
  $n$--covers $a$, but is not contained in the interior of an indivisible Nielsen
  path. 
  
  We can express $\alpha$ as a concatenation $\alpha_1 \alpha_2 \ldots
  \alpha_r$ of indivisible Nielsen paths $\alpha_i$,
  $i=1,\ldots,r$. The $g$-fixed points of $\alpha$ are precisely the
  endpoints of the $\alpha_i$'s. Then since $\alpha_0$ is not
  contained in an indivisible Nielsen path, it must contain one of one
  of these fixed points $p$. It therefore contains at least $n$ copies
  of the fixed point $p$. Therefore some sequence $\alpha_i \ldots
  \alpha_j$ forms a Nielsen path and corresponds to the
  conjugacy class of $a$. A closed Nielsen path corresponds to a
  periodic loop. It follows that $\ell_{T_+}(a) = 0$, contradicting
  the hypothesis on $a$.
  \end{proof}

A path $\alpha \subset G$ is a \emph{vanishing path} of $g$ if $[g^m(\alpha)]$ is trivial (i.e., is a point) for some $m \geq 1$. We record the following observation from \cite{ar:BBC10}: 

\begin{lemma}\label{lem:vp-in-Np}
Let $\phi$ and $g\co G \to G$ be as in Theorem \ref{th:Np-or-vp}, and suppose that $\alpha$ is an indivisible Nielsen path. Then any subpath $\beta \subseteq \alpha \dagger_\epsilon$ for some $\epsilon >0$ 
is contained in a vanishing path. 
\end{lemma}

\begin{proof}
  An indivisible Nielsen path can be decomposed into a sequence of legal paths as $\alpha = \alpha_0
  \cdot \beta_0 \cdot \overline{\beta}_1 \cdot \overline{\alpha}_1$,
  where $g(\alpha_i) = \alpha_i \cdot \beta_i$ and $g(\beta_0) =
  g(\beta_1)$ \cite{ar:BH92}. Hence for $\epsilon >0$, with large enough $n$, the path
  $[g^n(\alpha\dagger_\epsilon)]$ is contained in $\beta_0 \cdot
  \overline{\beta}_1$, and hence $[g^{n+1}(\alpha\dagger_\epsilon)]$ is
  trivial. 
\end{proof}

Putting together Theorem \ref{th:Np-or-vp} and Lemmas \ref{lem:onlyiNP} and \ref{lem:vp-in-Np}, we have the following: 

\begin{proposition}\label{prop:find-vp}
  Suppose that $\phi \in \Out F_k$ is fully irreducible with stable tree $T_+=T_+(\phi)$, 
  that $g
  \co G \to G$ is a train-track representative of $\phi$, and that $a\in F_k$ is such
  that $\ell_{T_+}(a) > 0$.  Then there is a surjective
  $F_k$--equivariant map $f_g \co \widetilde{G} \to T_+$ such that if
  $[x,y] \subset \widetilde{G}$ is a lift of a reduced path $\alpha
  \subset G$ that $n$--covers $a$, for some $n \geq 2$, and if $f_g(x) =
  f_g(y)$, then $\alpha$ contains a vanishing path that $n$--covers
  $a$.
\end{proposition}


\subsection{\texorpdfstring{Laminations and the map $\CQ$}{Laminations and the map Q}}\label{ssc:lamination-Q}

There are several notions of a ``lamination'' on a free group.
For a full discussion on three different approaches and the
relations between them, see \cite{ar:CHL08-I,ar:CHL08-II,ar:CHL08-III}.
We will only briefly describe the elements of the theory we need here.

The group $F_k$ is hyperbolic and hence has a boundary $\bd F_k$. We denote:
\[ \bd^2 F_k = \{ (x_1,x_2) \in \bd F_k \times \bd F_k \; | \; x_1
\neq x_2 \} \] 
This set is naturally identified with the space of oriented
\emph{bi-infinite geodesics} in a tree $T \in cv_k$ as we explain now.

An oriented bi-infinite geodesic is an isometric embedding $\ell \co
\BR \to T$ considered up to reparametrization preserving the
orientation.  Any geodesic has two distinct endpoints in $\bd T$,
denoted $\ell(\infty)$ and $\ell(-\infty)$.  We can thus identify the
geodesic $\ell$ with endpoints $(\ell(\infty),\ell(-\infty)) \in \bd^2
T = \{ (x_1,x_2) \in \bd T \times \bd T \; | \; x_1 \neq x_2 \}$,
which, via the action of $F_k$, is naturally identified with $\bd^2
F_k$.  Conversely, a point $(x_1,x_2) \in \bd^2 F_k$ determines two
distinct points $x_1',x_2' \in \bd T$.  Between these two points,
there is a unique (up to orientation preserving reparametrization)
oriented geodesic $\ell \co \BR \to T$ such that $\ell(\infty) = x_1'$
and $\ell(-\infty) = x_2'$.  We will use this identification without
further remark.

There is fixed point free involution on $\bd^2 F_k$ defined by
$\sigma\co (x_1,x_2) \to (x_2,x_1)$, corresponding to reversing a
geodesic's orientation in $T$.

A \emph{lamination} is a closed $F_k$--invariant and
$\sigma$--invariant subset $\Lambda \subseteq \bd^2 F_k$. The set of
algebraic laminations inherits a Hausdorff topology from $\bd^2 F_k$,
which is described in \cite{ar:CHL08-I}.  A nontrivial element $a \in
F_k$ determines a {\it minimal rational} lamination:
$$\Lambda(a) = 
\{ (ga^{-\infty}, ga^{+\infty}) \cup (ga^{+\infty}, ga^{-\infty}) \ |
\ g \in F_k \} $$ Note that the set $\Lambda(a)$ depends only on the
conjugacy class of $a$. Although we will not need them here, we
mention that the set of {\it rational} laminations consists of finite
unions of minimal rational laminations. The most important example of
a lamination in what follows is the \emph{stable lamination}
$\Lambda_+(\phi)$ associated to a fully irreducible
element $\phi \in \Out F_k$, as defined in
\cite{ar:BFH97}.\footnote{Note that in \cite{un:HM}, the stable
  lamination is called the ``expanding lamination'' and denoted by
  $\Lambda_-$ as it is more naturally associated to $T_-(\phi)$.  See
  Proposition \ref{prop:stableQ}.}  The \emph{unstable lamination}
$\Lambda_-(\phi)$ associated to $\phi$ is the stable
lamination of $\phi^{-1}$, so that $\Lambda_-(\phi) =
\Lambda_+(\phi^{-1})$.

Let $g \co G \to G$ be a train-track representative of $\phi$ with
expansion factor $\lambda$.  After passing to a power of $g$ if
necessary, we can assume that $g$ has a fixed point $x$ contained in
the interior of an edge.  For some small $\epsilon$--neighborhood $U$
of $x$, we have that $g(U) \supset U$.  Fix an isometry $\ell \co
(-\epsilon,\epsilon) \to U$ and extend this to the unique isometric
immersion $\ell \co \BR \to G$ such that $\ell(\lambda^n t) = g^n(
t)$.  This immersion lifts to a collection of geodesics $\tilde{\ell}
\co \BR \to \widetilde{G}$.  Using the identification mentioned above
between $\bd^2 F_k$ and the space of geodesics in $\widetilde{G}$, the
collection of all geodesics (called \emph{leaves}) constructed as
above determines a closed $F_k$--invariant subset of $\bd^2 F_k$
called the \emph{stable lamination}.  It is proved in \cite{ar:BFH97}
that this set is well-defined independent of $g$. The leaves of
$\Lambda_+(\phi)$ are \emph{quasi-periodic} \cite{ar:BFH97}, so that for
every $L > 0$ there is an $L' > L$ such that for every interval $I$ of
length $L$ and every interval $I'$ of length $L'$ there is an element
$x \in F_k$ such that $x\ell(I) \subseteq \ell(I')$.

Given a basis $\CA$ of $F_k$ and a tree $T \in \overline{cv}_k$,
define the set $L^1_\CA(T)$ as the set of right infinite reduced words
$x = x_1x_2x_3\cdots$ in the basis $\CA$ such that for some $p \in T$,
the sequence of points $ (x_1x_2 \cdots x_i)p$ is bounded. The
identification of right infinite reduced words in $\CA$ with $\bd F_k$
identifies $L^1_\CA(T)$ with a subset $L^1(T) \subseteq \bd F_k$ that
is well-defined independent of the choice of basis. Bounded backtracking ensures the existence of a well-defined injective map $\CQ \co \bd F_k - L^1(T) \to \bd T$. 
Using the injectivity of $\CQ$ on $\bd F_k - L^1(T)$, we associate to any 
oriented bi-infinite geodesic $\alpha = (x_1,x_2) \in
\bd^2 F_k - (L^1(T))^2$ an oriented bi-infinite
geodesic $\alpha_T \subset T$, namely $\alpha_T = (\CQ(x_1),\CQ(x_2))
\in \bd^2 T$, if neither endpoint of $\alpha$ is in
$L^1(T)$; otherwise we define $\alpha_T$ to be the empty set. In the latter case, following 
\cite{ar:LL03}, we say that the geodesic $\alpha$ is \emph{$T$--bounded}.

In certain cases, the map $\CQ\co \bd F_k - L^1(T) \to \bd T$ extends to a map on $\bd F_k$.

\begin{proposition}[\cite{ar:LL03}, Proposition 3.1]\label{prop:Q}
  Suppose $T \in \overline{cv}_k$ has dense orbits and trivial arc
  stabilizers (e.g., the stable tree for a fully irreducible outer
  automorphism). There exists a map $\CQ\co L^1(T) \to \overline{T}$
  to the metric closure $\overline{T}$ of $T$ such that, for any $f
  \co T_0 \to T$, where $T_0 \in cv_k$, and any ray $\rho$ in $T_0$
  representing $x \in L^1(T)$, the point $\CQ(x)$ belongs to the
  closure of $f(\rho)$ in
  $\overline{T}$. 
  \end{proposition} 

\noindent Combining this with the previous discussion, we have a
map $\CQ \co \bd F_k \to \overline{T} \cup \bd T$ whenever $T$ has
dense orbits and trivial arc stabilizers.  

The relation between stable trees and laminations is illustrated by
the following.

\begin{proposition}[\cite{ar:BFH97}, Lemma 3.5 (3) \& \cite{ar:LL03},
  Corollary 2.3]\label{prop:stableQ}
  Suppose that $\phi \in \Out F_k$ is fully irreducible with stable tree
  $T_+$ and unstable lamination $\Lambda_-$.  Let $\CQ \co \bd F_k \to
  \overline{T} \cup \bd T$ be the map defined following Proposition \ref{prop:Q}.
  Then for any leaf $\ell \in \Lambda_-$, we have $\CQ(\ell(\infty)) =
  \CQ(\ell(-\infty))$.
\end{proposition}

\noindent In light of the above propositions, we can define for any tree $T \in
\overline{cv}_k$ with dense orbits and trivial arc stabilizers
\cite{ar:CHL08-II}:
\[ 
L^2_\CQ(T) = \{ (x_1,x_2) \in \bd^2 F_k \; | \; \CQ(x_1) = \CQ(x_2) \}
\] where $\CQ \co \bd F_k \to \overline{T} \cup \bd T$ is the map from
Proposition \ref{prop:Q}.  With this definition, Proposition
\ref{prop:stableQ} states that $\Lambda_-(\phi) \subseteq
L^2_{\CQ}(T_+^\phi)$ where $\phi \in \Out F_k$ is fully irreducible.
If $\phi$ is hyperbolic (i.e.,~$\phi$ does not have nontrivial periodic
conjugacy class) then $L^2_{\CQ}(T_+^\phi)$ is the
``diagonal closure'' of $\Lambda_-(\phi)$ \cite{un:KL}.

\bigskip

Missing from the above is a discussion of \emph{measured geodesic
  currents} and \emph{Dehn twist automorphisms} needed for Section
\ref{sc:example}.  We defer their discussion until needed.


\section{Geometric relative twisting}\label{sc:gtwist}

Our first definition of relative twisting for $\Out F_k$ follows
closely the original geometric notion for the mapping class group,
upon replacing a surface with a suitable 2--complex. This complex, the
{\it Guirardel Core} for two $F_k$--trees $T,T'$, is a certain
$F_k$--invariant subspace $\CC \subset T \times T'$ (with the diagonal
action). We will not need the precise definition of the complex for
our purposes; rather, we record in Proposition \ref{prop:core} just those
properties of $\CC$ we do require, together with references.

In the following, if $p$ is a point in $T$, then $\CC_{p} = \{ x' \in
T' \, | \, (p,x') \in \CC \}$; similarly, for $p' \in T'$, we have $\CC_{p'} = \{ x \in T \, |
\,(x,p') \in \CC \}$.  These sets are called the \emph{slices of the
  core}. 

\begin{proposition}\label{prop:core}
 Suppose $T,T' \in CV_k$.
 \begin{enumerate}

 \item The core $\CC \subset T \times T'$ is nonempty, connected,
   closed, {\rm CAT(0)}, $F_k$--invariant and has convex fibers, i.e., 
   the slices $\CC_p$ and $\CC_{p'}$ are each convex for all $p \in T$ and $p'
   \in T'$. Moreover, $\CC$ is the minimal (with respect to inclusion)
   subset of $T \times T'$ with these properties \cite[Main
   Theorem]{ar:Gu05}.

 \item The quotient $\CC/F_k$ has finite volume \cite[Theorem
   8.1]{ar:Gu05}. This volume is called the \emph{intersection
     number}, denoted $i(T,T')$.

 \item For any $p' \in T'$ that is not a vertex, any arc $\gamma
   \subset \CC_{p'}$ is contained in a vanishing path of any change of
   marking map $f\co T \to T'$ \cite[Lemma 3.7 \& Remark
   5.3]{ar:BBC10}.

 \end{enumerate}
\end{proposition}
\noindent For the complete definition of the core, along with
examples, see \cite{ar:BBC10,ar:Gu05}.

Before going further we briefly recall relative twisting for curves on
a surface. Let $S$ be a surface of genus at least two, equipped with a
hyperbolic metric. We can consider $\pi_1(S)$ as a discrete group of
isometries of $\mathbb{H}^2$, so that $S = \mathbb{H}^2/\pi_1(S)$. Fix
three simple closed curves, $\alpha,\beta,\gamma$, so that $\beta$ and
$\gamma$ both intersect $\alpha$. Each of these curves corresponds to
a conjugacy class of an element in $\pi_1(S)$, and we can assume that
all three are geodesics on $S$. Let $S_\alpha$ be an annular cover of
$S$ corresponding to $\alpha$; that is, the quotient of $\mathbb{H}^2$
by the cyclic group generated by a representative of $\alpha$ in
$\pi_1(S)$. We let $\alpha_c$ denote the unique lift of $\alpha$ to
$S_\alpha$ that is closed.  The \emph{twist of $\beta$ and $\gamma$
  relative to $\alpha$} is defined as the maximum geometric
intersection number between $\beta'$ and $\gamma'$ that intersect
$\alpha_c \subset S_\alpha$, where $\beta'$ and $\gamma'$ range over
lifts of $\beta$ and $\gamma$ to $S_\alpha$.

We can reformulate this in terms of the universal cover
$\widetilde{S}$, defining the relative twist as follows. Fixing a lift
$\tilde{\alpha}$ of $\alpha$ to $\tilde{S}$, the twist of $\beta$ and
$\gamma$ relative to $\alpha$ is the maximum number of
$\alpha$--translates of $\tilde{\gamma}$ that intersect
$\tilde{\beta}$, over all choices of lifts $\tilde{\beta},
\tilde{\gamma}$ of $\beta,\gamma$ that intersect $\tilde{\alpha}$. See
Figure \ref{fig:rtwist}. This interpretation can be extended to trees
in Outer space, using the Guirardel core of $F_k$--trees $T$ and $T'$
in place of $\tilde{S}$.

\begin{figure}[t]
 \labellist
 \small\hair 2pt
 \pinlabel $\tilde{\alpha}$ [l] at 132 220
 \pinlabel $\tilde{\beta}$ [l] at 159 180
 \pinlabel $\tilde{\gamma}$ [l] at 41 154
 \pinlabel $\alpha\tilde{\gamma}$ [l] at 52 208
 \pinlabel $\alpha^2\tilde{\gamma}$ [l] at 81 244
 \pinlabel $\alpha^3\tilde{\gamma}$ [l] at 110 265
 \pinlabel $\alpha^{-1}\tilde{\gamma}$ [l] at 56 92
 \pinlabel $\alpha^{-2}\tilde{\gamma}$ [l] at 192 47
 \pinlabel $\alpha^{-3}\tilde{\gamma}$ [l] at 169 21
 \endlabellist
 \centering
 \includegraphics[scale=0.9]{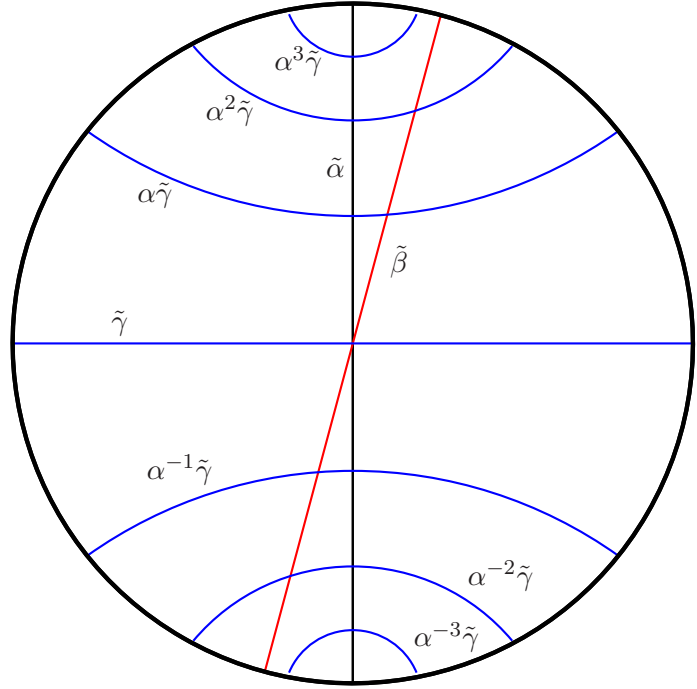}
 \caption{The relative twist of $\beta$ and $\gamma$ with respect to
   $\alpha$ is 5.}\label{fig:rtwist}
\end{figure}

The role of the simple closed curves $\beta,\gamma$ is filled by
tracks on $\CC$, and of the simple closed curve $\alpha$ by the axis
of an element of $F_k$ in $\CC$. A {\it track for $T$} in $\CC$ is the
set $\{ p\} \times \CC_{p}$
where $p$ is the midpoint of some edge of $T$; a track for $T'$ is
defined similarly.  We will also use {\it track} to refer to the image
of a track in the quotient $\CC/F_k$.  We record some elementary
properties of tracks.

\begin{enumerate}

\item Every track separates $\CC$.

\item Every track is a finite subtree.

\item Every track is a convex subset of $\CC$.

\end{enumerate}

As $\CC$ is CAT(0), and every nontrivial element acts hyperbolically,
the minset of a nontrivial element $g \in F_k$ is isometric to a
product $Y \times \BR$, where $Y$ is a convex subset of $\CC$
\cite{bk:BH99}. An {\it axis} of a nontrivial element $a \in F_k$ is a subset
of the minset of $a$ of the form $\{ y_0 \} \times \BR$.  The element
$a$ acts by translation on any of its axes.

Lemmas \ref{lem:gtwist-wd} and \ref{lem:wd} describe the extent to which
the intersection number between tracks and axes is well-defined.

\begin{lemma}\label{lem:gtwist-wd}
  Let $a \in F_k$ be a nontrivial element, $T,T' \in CV_k$ and
  consider the core $\CC \subset T \times T'$.  If a track $\tau$ in
  $\CC$ intersects an axis of $a$ in $\CC$, then it intersects every
  axis of $a$ in $\CC$.
\end{lemma}

\begin{proof}
  Let $\tau = \{ p \} \times \CC_p$ be a track for $T$ that intersects
  an axis for $a$.  Let $Y \times \BR \subset \CC$ be the minset for
  $a$.  As tracks and axes are convex, their intersection is
  convex as well, and hence connected.  Moreover, as tracks are finite,
  there are $s,t \in \BR$ such that $\left(Y \times \BR\right) \cap
  \tau \subset Y \times [s,t]$.

 Let  $\{ y_0 \} \times \BR \subset Y \times \BR$ be an axis of $a$ that intersects $\tau$. 
  As $\left(\{ y_0 \} \times \BR\right) \cap \tau$ is connected, we
  have that $\{ y_0 \} \times (-\infty,s)$ and $\{y_0 \} \times (t
  ,\infty)$ project to different components of $T - \{ p \}$.  Hence
  $\tau$ separates $\{ y_0 \} \times \{ s-1 \}$ from $\{ y_0 \} \times
  \{ t+1 \}$ in $\CC$.  If there were an axis, say $\{ y_1 \} \times
  \BR$, that did not intersect $\tau$, then the concatenation of the
  fiber-wise paths
  \[\{y_0 \} \times \{ s-1 \} \to \{y_1 \} \times \{ s-1 \} \to \{y_1
  \} \times \{ t+1 \} \to \{y_0 \} \times \{ t+1 \}\] would be a path
  that connected $\{ y_0 \} \times \{ s-1 \}$ to $\{y_0 \} \times \{
  t+1 \}$ avoiding $\tau$, which is a contradiction as $\tau$
  separates these points.  Thus every axis for $a$ intersects $\tau$.
\end{proof}

\begin{lemma}\label{lem:wd}
  Let $a \in F_k$ be a nontrivial element, $T,T' \in CV_k$ and
  consider the core $\CC \subset T \times T'$.  Fix a track $\tau'$
  for $T'$ that intersects an axis of $a$ in $\CC$.  Let $\tau_0$ and
  $\tau_1$ be two tracks for $T$ that intersect an axis of $a$.  Then:
 \[ \Bigl| |\tau' \cap \la a \ra \tau_0 | - |\tau' \cap \la a \ra
 \tau_1 | \Bigr| \leq 1. \]
\end{lemma}

\begin{proof}
   Suppose $|\tau' \cap \la a
  \ra \tau_0| = n$.  We will show that $|\tau' \cap \la a \ra \tau_1|
  \geq n-1$.  The statement of the lemma follows after interchanging
  $\tau_0$ and $\tau_1$ and applying the same argument.

  Since a track for $T$ and a track for $T'$ can intersect at most
  once in $\CC$, the track $\tau'$ intersects exactly $n$ $\la a
  \ra$--translates of $\tau_0$. We claim that there is an $i$ such
  that $\tau'$ intersects $a^{i+j} \tau_0$ for $j = 0,\ldots, n-1$.
  Indeed, this follows as $\tau'$ is connected, and as $a^m \tau_0$
  separates $a^{m-r}\tau_0$ from $a^{m+s}\tau_0$ for all $m$ and
  positive $r$ and $s$.  Replacing $\tau_0$ by $a^i\tau_0$, we can
  assume that $\tau'$ intersects $\tau_0,\ldots a^{n-1}\tau_0$.
  
  If $\tau_1 = a^i\tau_0$ for some $i$, then the statement is obvious. 
  Otherwise, as $|\tau' \cap \la a \ra\tau_1 |$ only depends on the
  orbit of $\tau_1$ under $a$, we can replace $\tau_1$ by $a^i \tau_1$
  for some $i$ to assume that $\tau_1$ separates $\tau_0$ from
  $a\tau_0$.

 We claim that $\tau'$ intersects $\tau_1,\ldots, a^{n-2}\tau_1$.
 Indeed, since $\tau'$ is connected and since $a^i \tau_1$ separates
 $a^i\tau_0$ from $a^{i+1}\tau_0$, both of which intersect $\tau'$ for
 $i = 0,\ldots n-2$, the track $\tau'$ must intersect $a^i \tau_1$.
\end{proof}

As $|\tau' \cap \la a \ra \tau| = |\la a \ra \tau' \cap \tau|$, Lemma
\ref{lem:wd} shows that if $\tau'_0$ and $\tau'_1$ are tracks for $T'$
that intersect an axis of $a$, and likewise $\tau_0$ and $\tau_1$ are
tracks for $T$ that intersect an axis of $a$, then:
\[ \Bigl| |\tau_0' \cap \la a \ra \tau_0 | - |\tau_1' \cap \la a \ra
\tau_1 | \Bigr| \leq 2. \] With this bound we can define the relative
twist number.  By $Tr_a(T)$ we denote the set of tracks for $T$ in
$\CC$ that intersect an (and hence every) axis of $a$.  We
define the set $Tr_a(T')$ similarly.

\begin{definition}\label{def:gtwist}
  Given $T,T' \in CV_k$ and a nontrivial element $a \in F_k$, define
  the \emph{twist of $T$ and $T'$ relative to $a$} as:
\begin{equation*}
  \tau_a(T,T') = \max_{\tau' \in Tr_a(T'),\tau \in Tr_a(T)}
  |\tau' \cap \la a \ra \tau|.
\end{equation*}
\end{definition}

We remark that this number is always finite.  Indeed, as tracks are
finite, the quantities we are maximizing over are finite. Then by the
above discussion, there are only finitely many distinct possibilities
for these numbers.

The significance of the relative twist number to the geodesic in
$CV_k$ connecting two marked graphs is the following, to be used in
Section \ref{sc:smallcycle} to prove Theorem \ref{th:geo-twist}:

\begin{proposition}\label{prop:ntwist-vp}
  Suppose $G,G' \in CV_k$ with $d = d_L(G,G')$ such that $\tau_a(G,G')
  \geq n$ for some nontrivial $a \in F_k$.  Then for every change of
  marking map $g \co G \to G'$, there is a vanishing path $\gamma
  \subset G$ that $n$--covers $a$.
\end{proposition}

\begin{proof}
  Let $\widetilde{G}$ and $\widetilde{G}'$ be the universal covers of
  $G$ and $G'$ respectively, and consider the core $\CC \subset
  \widetilde{G} \times \widetilde{G}'$.  Fix an axis of $a$, and
  tracks $\tau \in Tr_a(\widetilde{G}), \, \tau' \in
  Tr_a(\widetilde{G}')$ such that $\tau_a(G,G') = |\tau' \cap \la
  a \ra\tau|$.

  Let $m_0$ and $m_1$ be the least and greatest integer, respectively,
  such that $\tau' \cap a^{m_0}\tau \neq \emptyset$ and $\tau' \cap a^{m_1}\tau \neq
  \emptyset$.  Thus $\tau_a(G,G') = m_1 - m_0 \geq n$.  Denote $x_0 =
  \tau' \cap a^{m_0}\tau$ and $x_1 = \tau' \cap a^{m_1}\tau$ and let
  $\rho$ be the path in $\tau'$ connecting $x_0$ to $x_1$.  As $\tau'
  = \CC_{p'} \times \{ p' \}$ for some point $p' \in \widetilde{G}'$,
  we can consider $\rho$ as a path in $\CC_{p'} \subset
  \widetilde{G}$. Notice that the endpoints of $\rho$, also denoted $x_0$ and $x_1$, are on the
  axis for $a$, and that $a^{m_1-m_0}x_0 = x_1$.  
  Thus $a$ is $n$--covered by $\rho$. By Proposition \ref{prop:core}(3), the path $\rho$ is contained in a
  vanishing path $\gamma$ for any change of marking map $G \to G'$.
  As $\gamma$ contains $\rho$, the vanishing path $\gamma$ $n$--covers
  $a$ as well.
\end{proof}


\section{Algebraic relative twisting}\label{sc:atwist}

In this section we give our algebraic interpretation of relative twisting and develop some consequences that will be useful for
applications in later sections.  The key result is Proposition
\ref{prop:ncover-vp}, which is used to prove Theorem
\ref{th:alg-twist}.

\begin{definition}\label{def:woMeasure}
  Given $T \in \overline{cv}_k$, a lamination $\Lambda
  \subset \partial^2 F_k$, and an element $a \in F_k$, if
  $\ell_T(a) \neq 0$, then we define the \emph{twist of $T$ and
    $\Lambda$ relative to a} to be:
\begin{equation*}
  \tau_a(T,\Lambda) = \sup \left\{  
        \frac{L_T(\alpha_T \cap T^{\la a \ra})}{\ell_T(a)} \, \Big| \, 
        \alpha \in \Lambda \right\}.
\end{equation*}
If $\ell_T(a) = 0$, then define $\tau_a(T,\Lambda) = 0$.
\end{definition}

Recall that given $\alpha = (x_1,x_2) \in \bd^2 F_k$, we have that $\alpha_T$ is equal to 
$(\CQ(x_1),\CQ(x_2)) \in \bd^2 T$ if $\alpha$ is not $T$--bounded, and is equal to 
the empty set otherwise.  We insist that $L_T(\emptyset) = 0$.

\begin{remark}
  We allow for the possibility that $\tau_a(T,\Lambda) = \infty$.
  This occurs in particular for the rational lamination $\Lambda(a)$,
  when $\ell_T(a) \neq 0$.
\end{remark}

Central to our analysis is the following proposition: 

\begin{proposition}\label{prop:lowersemicontinuous}
  Suppose $a \in F_k$, $T \in \overline{cv}_k$, and that $\Lambda$ is
  a lamination containing no $T$--bounded geodesic. If $\{ T_i \}$ is
  a sequence of trees in $cv_k$ converging to $T$, and $\{ \Lambda_i
  \}$ is a sequence of laminations converging to $\Lambda$, then:
  \[ \lim_{i \to \infty} \tau_a(T_i,\Lambda_i) \geq \tau_a(T,\Lambda). \]
\end{proposition}

\begin{proof}
  The proposition is obviously true when $\ell_T(a) = 0$, and so we
  assume that $\ell_T(a) > 0$.
  
  We proceed with the following:
  \begin{claim*}
    If $T \in \overline{cv}_k$, and if $\alpha, \beta \in \bd^2 F_k$
    are not $T$--bounded and
    $\alpha(\infty),\alpha(-\infty),\beta(\infty),\beta(-\infty)$ are
    four distinct points, then for sufficiently close $T' \in cv_k$,
    we have $L_{T'}(\alpha_{T'} \cap \beta_{T'})$ close to
    $L_T(\alpha_T \cap \beta_T)$.
  \end{claim*}
  \begin{proof}
    As $\CQ$ is injective on $\bd F_k - L^1(T)$, we have that
    $\alpha_T \cap \beta_T$ is compact set.
    
    Fix a tree $T_0 \in cv_k$, a map $f\co T_0 \to T$ that is linear
    on edges, and elements $a_i^\pm \in T_0$ so that $[a_i^-,a_i^+]
    \to \alpha_{T_0}$. Then $\alpha_i = [f(a_i^-),f(a_i^+)] \to
    \alpha_T$; in particular, the overlap of $\alpha_i$ and $\alpha_T$
    can be made arbitrarily large.  Similarly define $b_i^\pm \in T_0$
    and $\beta_i = [f(b_i^-),f(b_i^+)]$ so that, as before, we have
    $\beta_i \to \beta_T$.

    Fix a $T' \in cv_k$ and an equivariant map $f' \co T_0 \to T'$,
    linear on edges.  As before, we have $\alpha'_i =
    [f'(a_i^-),f'(a_i^+)] \to \alpha_{T'}$ and $\beta'_i =
    [f'(b_i^-),f'(b_i^+)] \to \beta_{T'}$.

    Now choose $n$ large enough so that each of $\alpha_n \cap
    \alpha_T$ and $\beta_n \cap \beta_T$ contains $\alpha_T \cap
    \beta_T$.  Then increase $n$ if necessary so that $\alpha_n' \cap
    \alpha_{T'}$ and $\beta_n' \cap \beta_{T'}$ each contain
    $\alpha_{T'} \cap \beta_{T'}$.  For sufficiently close trees $T,
    T',$ the lengths $L_T(\alpha_n \cap \beta_n)$ and
    $L_{T'}(\alpha'_n \cap \beta'_n)$ of the overlaps are close
    \cite{ar:Clay05,ar:GL07}.  By choice of $n$, we have $L_T(\alpha_n
    \cap \beta_n) = L_T(\alpha_T \cap \beta_T)$ and $L_{T'}(\alpha'_n
    \cap \beta'_n) = L_{T'}(\alpha_{T'} \cap \beta_{T'})$.  The claim
    follows.
  \end{proof}

  We are now prepared to complete the proof of the proposition.  First
  assume that $\tau_a(T,\Lambda) \neq \infty$.  This implies that no
  geodesic in $\Lambda$ has $a^{+\infty}$ or $a^{-\infty}$ as an
  endpoint and so we can use the above claim.  Let $\epsilon$ be small
  and choose a geodesic $\alpha \in \Lambda$ so that:
  $$L_T(\alpha_T \cap T^{\la a \ra})/\ell_T(a) > \tau_a(T,\Lambda) - \epsilon.$$ 
  Then by the above:
  \[ \frac{L_{T'}(\alpha_{T'} \cap T'^{\la a \ra})}{\ell_{T'}(a)} >
  \frac{L_T(\alpha_T \cap T^{\la a \ra})}{\ell_T(a)} - \epsilon >
  \tau_a(T,\Lambda) - 2\epsilon \] for $T'$ sufficiently close to $T$.
  For $\Lambda'$ sufficiently close to $\Lambda$, there exists
  $\alpha' \in \Lambda'$ so that $\alpha_{T'} \cap T'^{\langle a
    \rangle} \subset \alpha'_{T'}$.  Thus:
  \[ \tau_a(T',\Lambda') \geq \frac{L_{T'}(\alpha_{T'} \cap T'^{\la a
      \ra})}{\ell_{T'}(a)} > \tau_a(T,\Lambda) - 2\epsilon\] for $T'$
  sufficiently close to $T$ and $\Lambda'$ sufficiently close to $\Lambda$.
  Since we obtain such an inequality for every $\epsilon$, the proposition holds.

  Suppose $\tau_a(T,\Lambda) = \infty$.  If $a^{+\infty}$ or
  $a^{-\infty}$ is an endpoint of a geodesic in $\Lambda$, then
  $\tau_a(T',\Lambda) = \infty$ for all trees $T' \in cv_k$.  Else, we
  have that for, for every $M >0$ there is a geodesic $\alpha \in
  \Lambda$ so that:
  $$\infty > L_T(\alpha_T \cap T^{\la a \ra})/\ell_T(a) > M.$$  
  Then arguing in a similar fashion as above, we have for $T'$
  sufficiently close to $T$ and $\Lambda'$ sufficiently close to
  $\Lambda$:
  \[ \tau_a(T',\Lambda') \geq \frac{L_{T'}(\alpha_{T'} \cap T'^{\la a
      \ra})}{\ell_{T'}(a)} > \frac{M}{2}. \] Since we obtain such an
  inequality for every $M$, the proposition holds.
\end{proof}
\begin{remark}\label{rm:T-bounded}
  Examples of tree, lamination pairs satisfying the hypotheses of
  Proposition \ref{prop:lowersemicontinuous} are:
  \begin{enumerate}

  \item $T_+(\phi),\Lambda_+(\phi)$ where $\phi \in \Out F_k$ is fully
    irreducible,

    \item $T,\Lambda$ where $T \in cv_k$ and $\Lambda$ is any
      lamination, and

  \item $T,\Lambda(a)$ whenever $\ell_T(a) \neq 0$. 

  \end{enumerate}
\end{remark}

For a fully irreducible element $\phi$ with large twist
$\tau_a(T_-(\phi),\Lambda_-(\phi))$ for some nontrivial $a \in F_k$,
our goal is to locate a train-track $G_0$ of $\phi$ with a vanishing
path that $n$--covers $a$, similar to Proposition
\ref{prop:ntwist-vp}.  Our tool to produce such a path is Proposition
\ref{prop:find-vp}.  First we see how to use the lamination to get the
required setup.

\begin{lemma}\label{lm:extend-segment}
  Suppose $\phi \in \Out F_k$ is fully irreducible, $g\co G \to G$ is
  a train-track representative for $\phi$, $T_+ \in \overline{cv}_k$
  is the stable tree for $\phi$, $f_g\co \widetilde{G} \to T_+$ is the
  induced map from Theorem \ref{th:Np-or-vp}, and $\ell\co \BR \to
  \widetilde{G}$ is a leaf of the unstable lamination $\Lambda_-$.
  Then for all $I \subset \BR$, there exists $I' = [x,y] \subset \BR$
  such that $I \subseteq I'$ and $f_g(\ell(x)) = f_g(\ell(y))$.
\end{lemma}

\begin{proof}
 We claim that for any $L \geq 0$, there is an interval
  $[a,b] \subseteq \BR$ such that $|b - a| \geq L$ and $f_g(\ell(a)) =
  f_g(\ell(b))$.  The lemma follows: by the quasi-periodicity of $\ell$, there is 
  then an interval $I_0
  = [a,b] \subset \BR$ such that $f_g(\ell(a)) = f_g(\ell(b))$ and
  $x\ell(I) \subseteq \ell(I_0)$.  Setting $I' =
  \ell^{-1}(x^{-1}\ell(I_0))$ 
  completes the proof. We must then just establish the claim. 

  Since $\ell$ is a leaf of the unstable lamination, we have
  $\CQ(\ell(-\infty)) = \CQ(\ell(\infty))$.  There are sequences $a_i
  \to -\infty$ and $b_i \to \infty$ such that $f_g(\ell(a_i)) \to
  \CQ(\ell(-\infty)) = \CQ(\ell(\infty))$ and $f_g(\ell(b_i)) \to
  \CQ(\ell(\infty))$ \cite[Lemma 3.4]{ar:LL03}.  Now we have two
  cases, either the sequences $\{f_g(\ell(a_i)) \}$ and $\{
  f_g(\ell(b_i)) \}$ are in the same component of $\overline{T}_+ - \{
  \CQ(\infty)\}$ or they are not.

  If the sequences are in the same component, choose $n$ with $|b_n -
  a_n| > L$. The arc $\alpha$ connecting $f_g(a_n)$ to $f_g(b_n)$ is
  then disjoint from $\CQ(\ell(\infty))$, and there is a unique point
  $p \in \alpha$ which is closest to $\CQ(\ell(\infty))$.  As
  $\overline{T}_+$ is an $\BR$--tree, $p$ is on the geodesic
  $[f_g(a_n),\CQ(\ell(-\infty))]$.  Then by continuity of $f\ell$,
  there is an $a' \leq a_n$ such that $f_g(a') = p$.  Likewise, there
  is a $b' \geq b_n$ such that $f_g(b') = p$.  Thus the inteveral
  $[a',b']$ satisfies the claim.

   Now suppose the sequences are not in the same component.  If
  $f_g(\ell(\BR))$ crosses the point $\CQ(\ell(\infty))$ infinitely
  many times, then we can find a sequence of points $a_i,b_i \in \BR$
  such that $f_g(\ell(a_i)) = f_g(\ell(b_i)) = \CQ(\ell(\infty))$ such
  that $|b_i - a_i| \to \infty$.  Indeed, there is a lower bound on
  the distance between two pre-images of $\CQ(\infty)$ in $\BR$ as
  every edge of $\widetilde{G}$ is isometrically embedded by $f_g$.  For
  large enough $i$, the interval $[a_i,b_i]$ satisfies the claim.

  Finally, suppose $f_g(\ell(\BR))$ crosses $\CQ(\ell(\BR))$ only
  finitely many times.  Let $a$ be the smallest number such that
  $f_g(\ell(a)) = \CQ(\ell(\infty))$.  Then arguing as in the first
  case using sequences $a_i \to -\infty$ and $b_i \to a$ ($b_i < a$)
  we can find the desired inteveral. 
\end{proof}

In the next proposition, we find a candidate vanishing path
in a train-track $G$ that folds over $a$ several times. The technicalities in its proof 
arise from the fact that, as the hypothesis concerns the {\it
  unstable} lamination, we must first find a large power of $a$
covered by a leaf of the lamination in a train-track map for the {\it
  inverse} $\phi^{-1}$ of $\phi$. Care is then required in mapping
this leaf over to a train-track for $\phi$, as there might be
excessive cancellation. We resolve this difficulty by applying powers
of $\phi^{-1}$, so that the length of $a$ dominates any such
cancellation.

\begin{proposition}\label{prop:unstable}
  Suppose that $\phi$ is fully irreducible with unstable tree $T_-$ and
 lamination $\Lambda_-$, with $\tau_a(T_-,\Lambda_-) \geq n + 2$ for some
  $a \in F_k$.  Then there exists a train-track graph $G \in CV_k$ for
  $\phi$ and a leaf of the unstable lamination $\ell \co \BR \to
  \widetilde{G}$ such that for all $L > 0$, there is a finite interval
  $I \subset \BR$ such that $[\ell(I)]\dagger_L$ $n$--covers $a$.
\end{proposition}

\begin{proof}
  Let $H \in CV_k$ be a train-track graph for $\phi\inv$ with train
  track representative $h \co H \to H$ and $G \in CV_k$ a train track
  graph for $\phi$ with train-track map $g\co G \to G$.  Fix Lipschitz
  homotopy equivalences $\kappa \co H \to G$ and $\kappa' \co G \to H$
  representing the change in markings.  Thus the following diagram is
  commutative up to homotopy:
  \[\xymatrix{G \ar[r]^g & G \ar[d]^{\kappa'} \\
    H \ar[u]^\kappa & H \ar[l]^h}\] Notice that $\kappa$ lifts to $\widetilde{\kappa}\co \widetilde{H}
  \to \widetilde{G}$, with bounded backtracking. In particular, we can pick a constant $C$ such
  that if a path $\gamma \subset H$ has length at least $C$, then
  $\kappa(\gamma)$ is not homotopically trivial relative to its
  endpoints.

  As $\tau_a(T_-,\Lambda_-) \geq n + 2$, Proposition
  \ref{prop:lowersemicontinuous} implies that
  $\tau_a(\widetilde{H}_0,\Lambda_-) \geq n + 2$, where $H_0 =
  H\phi^{-M}$ for some large $M$. Define $G_0 = G\phi^{-M}$.  Now
  there is a leaf $\ell \co \BR \to \widetilde{H}_0$ of the unstable
  lamination and an interval $I_0 \subset \BR$ such that $\ell(I_0)
  \subset \widetilde{H}_0^{\la a \ra}$ and
  $L_{\widetilde{H}_0}(\ell(I_0)) \geq (n +
  2)\ell_{\widetilde{H}_0}(a)$.  Notice that this implies that the
  loop representing the conjugacy class of $a$ in $H_0$ is legal with
  respect to $h$.  Then if we let $\lambda$ be the expansion factor
  for $\phi\inv$, and let $d = \ell_{\widetilde{H}_0}(a)$, we have
  $\ell_{\widetilde{H}_0\phi^{-m}}(a) = \lambda^md$.  Let $N$ be such
  that $\lambda^Nd \geq C$.  Define $G_1 = G_0\phi^{-N}$ and $H_1 =
  H_0\phi^{-N}$.

  Fix $L>0$ and let $L'$ be such that if we have paths $\gamma \subset
  \gamma'$ in $H_0$ and $\gamma$ has length at least $L'$ then the
  path $\kappa(\gamma)$ intersects $[\kappa(\gamma')]$ in a path of length
  at least $L$ (necessarily a subpath of $[\kappa(\gamma')]$, but not
  necessarily a subpath of $[\kappa(\gamma)]$).  Indeed, such an $L'$
  exists as $\kappa$ is a quasi-isometry and the graph $G_0$ has valence
  bounded by $2k$.  Finally, extend $I_0$ to $I = I_1 \cup I_0 \cup
  I_2 \subset \BR$ by including intervals $I_1$ and $I_2$ of length at
  least $(L' + C)/\lambda^N$.

  Notice that the lengths of $h^N(\ell(I_1))$ and $h^N(\ell(I_2))$ in
  $H_1$ are at least $L' + C$.  Moreover, the initial subsegment
  $\iota \subset h^N(\ell(I_1))$ of length $L'$ maps by $\kappa$ to a
  segment in $G_1$ that intersects $[\kappa h^N\ell(I)]$ in a path of
  length at least $L$, so that $\kappa(\iota)$ does not cancel with
  $\kappa h^N\ell(I_0)$, as $\iota$ and $h^N\ell(I_0)$ are separated in
  $h^N\ell(I)$ by a subsegment of length at least $C$.  Similarly for
  the terminal subsegment of $h^N(I_2)$.

  Recall $\ell(I_0)$ $(n+2)$--covers $a$ in $H_0$.  This clearly
  implies that $h^N\ell(I_0)$ $(n+2)$--covers $a$ in $H_1$ as all of
  the paths are legal with respect to $h$. Thus $h^N\ell(I_0)$
  contains a subpath whose image in $H_1$ represents the conjugacy
  class of $a^{n+2}$.  As
  $\ell_{\widetilde{H}_1}(a) \geq C$, a subsegment of length $n
  \ell_{\widetilde{G}_1}(a)$ in $[\kappa h^N\ell(I_0)]$ survives after
  tightening $[\kappa h^N\ell(I_1)] \cdot [\kappa h^N\ell(I_0)] \cdot
  [\kappa h^N\ell(I_2)]$ to get $[\kappa h^N\ell(I)]$.  Now tighten
  $\kappa\ell$ to get a leaf of the unstable lamination $\ell_1\co \BR
  \to \widetilde{G}_1$.  Thus
  $L_{\widetilde{G}_1}([\ell_1(I)]\dagger_L \cap \widetilde{G}_1^{\la
    a \ra}) \geq n\ell_{\widetilde{G}_1}(a)$ and hence
  $[\ell_1(I)]\dagger_L$ $n$--covers $a$.
\end{proof}

Proposition \ref{prop:ncover-vp} will now follow relatively quickly
once we observe the following consequence of Proposition
\ref{prop:unstable}.

\begin{corollary}\label{co:extend-point}
  Suppose that $\phi$ is fully irreducible with unstable tree $T_-$
  and lamination $\Lambda_-$, such that $\tau_a(T_-,\Lambda_-) \geq n
  + 2$ for some $a \in F_k$.  Then there exists a train-track $G \in
  CV_k$ for $\phi$, with train-track map $g\co G \to G$, and a path
  $\gamma = [x,y] \subset \widetilde{G}$ such that:
  \begin{enumerate}

  \item $\gamma$ $n$--covers $a$; and

  \item $f_g(x) = f_g(y)$ where $f_g \co\widetilde{G} \to T_+$ is the
    induced map (see Theorem \ref{th:Np-or-vp}) from $\widetilde{G}$ to $T_+$, the stable tree for
    $\phi$.
  \end{enumerate}
\end{corollary}

\begin{proof}
  Let $G$ be the train-track graph given by Proposition \ref{prop:unstable},
 and let $f_g \co \widetilde{G} \to T_+$ be the
  induced map.  Let $\ell \co \BR \to \widetilde{G}$ be the leaf of
  the unstable lamination and $I \subset \BR$ the interval produced by
  Proposition \ref{prop:unstable}, for $L = BBT(f_g) + 1$.  For this $I$, let $I'$ be the
  interval given by Lemma \ref{lm:extend-segment}.

  We claim that $\gamma = [\ell(I')]$ satisfies the conclusion of the
  corollary.  Indeed by Lemma \ref{lm:surviving-subsegment}, as
  $[\ell(I)]\dagger_L$ contains a path in the axis of $a$ of length
  $n\ell_{\widetilde{G}}(a)$, so does $\gamma = [\ell(I')]$.  By
  construction, $f_g$ identifies the endpoints of $\gamma$.
\end{proof}

Proposition \ref{prop:find-vp} applied to the train-track $G$ and the path $\gamma$ of Corollary \ref{co:extend-point} give the following:

\begin{proposition}\label{prop:ncover-vp}
  Suppose $\phi$ is a fully irreducible element with unstable tree
  $T_-$ and lamination $\Lambda_-$ such that $\tau_a(T_-,\Lambda_-)
  \geq n + 2$ for some $a \in F_k$. Then there exists a train-track $G
  \in CV_k$ for $\phi$, with train-track map $g\co G \to G$, and a
  vanishing path $\gamma \subset G$ that $n$--covers $a$.\qed
\end{proposition}


\section{Finding small cycles}\label{sc:smallcycle}

With our definition of relative twist, we can prove the analogue of a
special case of Rafi's characterization of short curves along
geodesics in Teichm\"uller space \cite{ar:Rafi05}. 

\begin{proposition}\label{prop:vpimpliessmallcycle}
  Suppose $G,G' \in CV_k$, $f\co G \to G'$ is a change of marking map
  with minimal slope, $d = d_L(G,G')$ and $a \in F_k$.  If there is a
  vanishing path $\gamma \subset G$ that $(n+2)$--covers $a$, then
  there is a geodesic $\alpha\co [0,d] \to CV_k$ such that $\alpha(0)
  = G$, $\alpha(d) = G'$ and for some $t \in [0,d]$, we have
  $\ell_{\alpha(t)}(a) \leq 1/n$.
\end{proposition}

\begin{proof} 
  Shrinking the edges of $G$ such that each edge is stretched by
  exactly $e^d$ results in a marked graph $G_0$ and provides a
  geodesic $\alpha\co [0,d_0] \to CV_k$ such that $\alpha(0) = G$,
  $\alpha(d_0) = G_0$ and $d = d_L(G_0,G') + d_0$. Denote the induced
  map $G_0 \to G'$ by $f$.

  Consider the graph $H_a = \widetilde{G}_0^{\la a \ra}/\la a \ra$ and
  the map $h_{G'}\co H_a \to G'$ which is the composition of the
  immersion $H_a \to G_0$ with the map $f\co G_0 \to G'$. By
  appropriately subdividing and folding the graph $H_a$, after pruning
  we obtain a graph immersion $H'_a \to G'$. Now choose a folding path
  based at $G_0$ whose folds correspond to the folds performed on
  $H_a$. The end of the folding path is a graph $G_1$ in which $H'_a$
  is immersed.

  Let $d_1 = d_L(G_0,G_1)$.  As $f$ has minimal slope, the above path
  extends $\alpha$ to a geodesic $\alpha\co [0,d_1] \to CV_k$ such
  that $\alpha(0) = G$ and $\alpha(d_1) = G_1$ \cite[Theorem
  5.5]{un:FM2}. Further $d_L(G_1,G') = d - d_1$.
  Denote the induced map $G_1 \to G'$ by $f_1$.

  The geodesic segment $\alpha$ can further be extended to a geodesic
  by folding $G_1$.  The image of $\gamma$ in $G_1$ (which we denote
  by $\gamma_1$) is a vanishing path for the map $f_1$.

  \begin{claim*}
    The path $\gamma_1 \subset G_1$ $n$--covers $a$.
  \end{claim*}
  \begin{proof}[Proof of Claim]
    Consider the graph $\widetilde{G}_0/\la a\ra$.  This graph
    consists of a collection of trees attached to $H_a$.  We consider
    the graph $H_a$ as oriented counterclockwise and decompose $H_a$
    into subsegments
    $\delta_1\epsilon_1\cdots\delta_\ell\epsilon_\ell$ where the
    $\delta_i$ are the maximal subsegments that remain upon folding
    $H_a \to H'_a$ and pruning.  The images of the $\epsilon_i$ are
    what get pruned.  There is a lift of $\gamma$ to
    $\widetilde{G}_0/\la a \ra$ that decomposes into subpaths $\gamma
    = \beta_0 \alpha \beta_1$ where $\beta_0$ and $\beta_1$ are
    embedded and $\alpha$ is the immersed path that covers $H_a$ $n+2$
    times.

    When folding the segments $\epsilon_i$, the initial part of
    $\beta_1$ may (by equivariance) become identified with some
    portion of $H_a$.  We are interested in bounding how much is
    identified with the terminal portion of $\alpha$ as this could
    reduce the amount of $\gamma_1$ that covers $a$.  We will show
    that the portion of $\alpha$ identified is a segment of $H_a$.
    In other words, we can reduce this amount by at most 1.

    Without loss of generality we assume that $\beta_1$ only
    intersects $H_a$ in a single vertex.  Suppose this vertex is in
    $\delta_i$ and consider performing the folds in $\epsilon_i$.
    After folding $\epsilon_i$, the subsegment of the terminal part of
    $\alpha$ identified with the initial part of $\beta_1$ either
    heads counterclockwise from $v_0$ (which we are not concerned with
    as this adds to the amount by which $\gamma_1$ covers $a$) or it
    is contained in the union of $\delta_i$ and $\epsilon_i$.  Indeed,
    we can just check locally that when folding two edges $e_1$ and
    $e_2$ in $\epsilon_i$ together in $H_a$ that $\beta_1$ cannot fold
    past (in the clockwise direction) the image of their terminal
    vertices.  This involves a few cases depending on the relative
    positions of $\beta_1$, $e_1$ and $e_2$; all of which are
    easily verified.

    Similarly, if $\beta_1$ only intersects $H_a$ in a vertex of
    $\epsilon_i$, then we find the the subsegment of the terminal
    portion of $\alpha$ that is identified with $\beta_1$ either heads
    counterclockwise from $v_0$ or it is contained in $\epsilon_i$.

    Thus when performing the folds in $H_a \to H'_a$, the initial
    portion of $\beta_1$ is identified at most one copy of $H_a$.
    Likewise, the same holds for the terminal portion of $\beta_0$.
    Therefore the image path $\gamma_1$ $n$--covers $a$.
  \end{proof}
  As a consequence of the claim, we have $\ell_{G'}(a^n) \leq BBT(f_1)$.
  Since $H'_a$ is immersed in every graph along the folding path between $G_1$
  and $G'$, we have $\ell_{G'}(a^n) = \Lip(f_1)\ell_{G_1}(a^n)$, so
  that
  \[
  \ell_{G_1}(a^n) = \ell_{G'}(a^n)/\Lip(f_1) \leq
  BBT(f_1)/\Lip(f_1) \leq 1
  \]
  and so $\ell_{G_1}(a) \leq \frac{1}{n}$.
\end{proof}

Combining Proposition \ref{prop:vpimpliessmallcycle} with Proposition
\ref{prop:ntwist-vp} we get the first of the main results of this
paper.

\begin{theorem}\label{th:geo-twist}
  Suppose $G,G' \in CV_k$ with $d = d_L(G,G')$ such that $\tau_a(G,G')
  \geq n+2$ for some $a$.  Then there is a geodesic $\alpha\co [0,d] \to
  CV_k$ such that $\alpha(0) = G$ and $\alpha(d) = G'$ and for some $t
  \in [0,d]$, we have $\ell_{\alpha(t)}(a) \leq 1/n$.  In other words,
  $\alpha([0,d]) \cap CV_k^{1/n}(a) \neq \emptyset$.
\end{theorem}

Additionally, combining Proposition \ref{prop:vpimpliessmallcycle}
with Proposition \ref{prop:ncover-vp} we get the second of the main
results of this paper.

\begin{theorem}\label{th:alg-twist}
  Suppose $\phi \in \Out F_k$ is fully irreducible, with unstable tree $T_-$ and
  lamination $\Lambda_-$ such that $\tau_a(T_-,\Lambda_-) \geq n+4$
  for some $a \in F_k$.  Then given any train-track $G$, there is an
  axis $\CL_\phi$ for $\phi$ that contains $G$ and a graph $G_0$ such
  that $\ell_{\widetilde{G}_0}(a) \leq 1/n$. In other words, $\CL_\phi
  \cap CV_k^{1/n}(a) \neq \emptyset$.
\end{theorem}


\section{Example}\label{sc:example}

Here we present an application of Theorem \ref{th:alg-twist} in which
we describe the asymptotic behavior of the translation length in
$CV_k$ of certain elements of $\Out F_k$, given as products $\phi_n =
\delta_1^n \delta_2^{-n}$ of powers of Dehn twists $\delta_1,
\delta_2$.  These types of outer automorphisms were considered in
\cite{ar:CP} and used in \cite{un:CP2} to show that there is no
homological obstruction to full irreducibility.  We briefly recall the
setup here; for more details consult either of the references
\cite{un:CP2, ar:CP}.


\subsection{Dehn twists}\label{ssc:Dehn-Twist}

A \emph{cyclic} tree is a Bass--Serre tree associated to a splitting
of $F_k$ over $\BZ$, either as an amalgamated free product or as an
HNN-extension.  To such a tree is an associated (outer) automorphism
called a \emph{Dehn twist}.  Given $F_k = A *_{\la c \ra} B$ we define
the Dehn twist automorphism $\delta_c$ of $F_k$ by:
\begin{align*}
  \forall a \in A \qquad & \delta_c(a) = a \\
  \forall b \in B \qquad & \delta_c(b) = cbc^{-1}.
\end{align*}
Likewise, given $F_k = A *_\BZ = \la A, t \ | \ t^{-1}ct = c' \ra$
for $c,c' \in A$, we define the Dehn twist $\delta_c$ of $F_k$ by:
\begin{align*}
  \forall a \in A \qquad & \delta_c(a) = a \\
  & \delta_c(t) = ct.
\end{align*}

Two cyclic trees $T_1$ and $T_2$ \emph{fill} if their associated Dehn
twists $\delta_1$, $\delta_2$ do not have any common invariant
conjugacy classes of proper free factors or cyclic subgroups.  As
mentioned above, if $T_1$ and $T_2$ fill, then for large enough $n$,
the element $\delta_1^n\delta_2^{-n}$ is fully irreducible
(and hyperbolic) \cite[Theorem~5.3]{ar:CP}.


\subsection{Currents}\label{ssc:currents}
A \emph{(measured geodesic) current} on $F_k$ is an
$F_k$--invariant and $\sigma$--invariant positive Radon measure on
$\bd^2 F_k$ (refer to Section \ref{ssc:lamination-Q}).  Such measures
where originally considered by Bonahon \cite{col:Bonahon91}, see also
\cite{col:Kapovich06}.  Given a tree $T \in CV_k$, there is an
identification between $\bd F_k$ and $\bd T$ used to interpret a
current as a measure on the set of (bi-infinite) geodesics in
$T$. Given a tight path $\alpha \subset T$, the \emph{two-sided
  cylinder} $Cyl_T(\alpha)$ is the collection of geodesics that
contain $\alpha$; such sets determine a basis for $\bd^2T$, and so in
turn for $\bd^2 F_k$.  When $\alpha$ is a fundamental domain for the
action of $a \in F_k$ on $T^{\la a \ra}$, we will denote
$Cyl_T(\alpha)$ by $Cyl_T(a)$. For a current $\nu \in Curr(F_k)$,
define $\la a, \nu \ra_T = \nu(Cyl_T(a))$.  As $\nu$ is
$F_k$--invariant, this is well-defined. The current is uniquely
defined by the values $\la a , \nu \ra_T$. If $c \in F_k$ is not a
proper power, then we define the {\it counting current} $\eta_c$ of
$c$ by:
\[ \la a, \eta_c \ra_T = \#\mbox{ of axes of conjugates of $c$ in }
Cyl_T(a) \] If $b = c^m$ where $c$ is not a proper power, then $\eta_b
= m\eta_c$.

Recall that an element $\phi \in \Out F_k$ is {\it hyperbolic} if it has no nontrivial periodic conjugacy classes in $F_k$; all such elements are necessarily non-geometric, in the sense that they are not induced by a surface homeomorphism. A hyperbolic fully irreducible element of $\Out F_k$ acts on the projectivized space of currents $\mathbb{P}Curr(F_k)$ with North--South dynamics \cite{thesis:Martin95}.  In particular, there are both \emph{stable} $[\mu_+(\phi)]$ and
\emph{unstable} $[\mu_-(\phi)]$ fixed projectivized currents associated to such an
element.  A similar statement holds for non-hyperbolic fully
irreducible elements as well, after restricting to the subspace of
$\mathbb{P}Curr(F_k)$ consisting just of those currents in the closure of the set of counting currents of primitive elements in $F_k$ \cite{KL_Boundary}.

The \emph{support} $Supp(\nu)$ of a current $\nu$ is the closure of the
union of all open sets $U$ such that $\nu(U) > 0$.  The support of a
current is a lamination.  The relationship between the stable currents and
stable laminations of a fully irreducible element of $\Out F_k$ is given by the  proposition below. The result is probably well-known, but to our knowledge, its proof does not appear in the literature. See
also \cite{un:KL} for closely related results.

\begin{proposition}\label{prop:lamination-current}
  Suppose $\phi \in \Out F_k$ is fully irreducible with stable and
  unstable laminations $\Lambda_+, \Lambda_-$ and stable and unstable
  currents $\mu_+, \mu_-$.  We have $Supp(\mu_\pm) = \Lambda_\pm$.
\end{proposition}

\begin{proof}
Let $g: G \to G$ be a train-track representative of $\phi$. Let $a \in F_k$ be a primitive element and $\alpha \subset G$ the reduced loop representing its conjugacy class. Then $\alpha$ is the union of $N$ legal paths in $G$ for some $N$, so that for all $m \geq 0$, the closed path $g^m(\alpha)$ consists of $N$ segments of leaves of the stable lamination $\Lambda_+$.

The set of cylinders $Cyl_{\tilde{G}}(\gamma)$, $\gamma$ a reduced path in $\tilde{G}$, not containing any leaf of $\Lambda_+$ give a cover of the complement of $\Lambda_+$. Choose one such cylinder $Cyl_{\tilde{G}}(\gamma)$, so that $\gamma$ is not a subsegment of any leaf of $\Lambda_+$. Then for any $m \geq 0$, the reduced loop $[g^m(\alpha)]$ contains at most $N$ copies of the image of $\gamma$ in $G$, and hence 
$\eta_{\phi^m(a)}(Cyl_{\tilde{G}}(\gamma)) \leq N$. Recall that, because $a$ was chosen to be primitive, we have the convergence of $[\eta_{\phi^m(a)}] \to [\mu_+]$. Now for a sequence $\lambda_m$ to give the convergence of 
$\frac{1}{\lambda_m} \eta_{\phi^m(a)} \to \mu_+$, it is necessary that
$\lambda_m \to \infty$ \cite[Theorem 1.2]{ar:KL10}. Thus we have $\mu_+(Cyl_{\tilde{G}}(\gamma)) = 0$.

We have shown that $Supp(\mu_+) \subseteq \Lambda_+$, a nonempty sublamination of a minimal lamination \cite{ar:BFH97,ar:KL10}. The claim of the proposition is verified.
\end{proof}


\subsection{Axes of products of Dehn twists}\label{ssc:example}

Let $k \geq 3$ and fix two filling cyclic trees $T_1$, $T_2$ with Dehn
twists $\delta_1$ and $\delta_2$.  Let $c_1$, $c_2$ denote the
respective edge stabilizers.  We assume that the set $\{c_1,c_2\}$ is
not \emph{separable}, i.e., no conjugates of $c_1$ and $c_2$ are
contained in a proper free factor of $F_k$, nor in complementary free
factors.  Further, we assume that $c_1$ and $c_2$ are not
\emph{simultaneously elliptic} in $\overline{CV}_k$, i.e.,
$\ell_T(c_1) + \ell_T(c_2) \neq 0$ for all $T \in \overline{cv}_k$.
These conditions can be guaranteed, for instance, by requiring $c_1$
and $c_2$ to be primitive elements sufficiently far apart in the free
factor complex \cite{un:CR}.

For the remainder of this section, elements $\phi_n \in \Out F_k$
denote the outer automorphisms induced by the automorphisms
$\delta_1^n\delta_2^{-n}$. For large enough $n$, the elements $\phi_n$
are fully irreducible and hyperbolic \cite[Theorem 5.3]{ar:CP}.  From
\cite[Theorem~5.2]{un:CP2}, we understand the limiting behavior of the
stable and unstable currents: $[\mu_+(\phi_n)] \to [\eta_{c_1}]$ and
$[\mu_-(\phi_n)] \to [\eta_{c_2}]$.
Using this, together with the parabolic behavior of Dehn twists on $\overline{CV}_k$
\cite{ar:CL95}, we show that the sequence of stable and unstable
trees likewise converge to the expected trees
(cf.,~\cite[Remark~5.3]{un:CP2}).

\begin{theorem}\label{th:treeconvergence}
  The trees $T_+(\phi_n) \in \overline{CV}_k$ converge to $T_2$.
  Similarly, the trees $T_-(\phi_n)$ converge to $T_1$.
\end{theorem}
\begin{proof}
  Denote $\psi_n = \delta_2^{-n} \delta_1^n$ so that $\phi_n =
  \delta_2^n \psi_n \delta_2^{-n}$.  Then as the outer automorphisms are
  conjugate by $\delta_2^n$, we have $T_+(\phi_n) =
  T_+(\psi_n)\delta_2^{-n}$.

 Recall that in \cite[Theorem~5.2]{un:CP2}, we determined that  $\lim_{n \to \infty}[\mu_-(\psi_n)] = [\eta_{c_1}]$.  The continuity of the 
  Kapovich--Lustig intersection form (see \cite{ar:KL10} for its definition and properties) implies that $c_1$ has a fixed point in an
  accumulation point of the sequence $\{T_+(\psi_n)\}$ (see
  \cite[Remark~5.3]{un:CP2}). 
Therefore as $c_1$ and $c_2$ are not simultaneously elliptic, $c_2$ has positive translation length in any such accumulation
  point.

  As $\overline{CV}_k$ is compact, some subsequence of $\{ T_+(\phi_n) \}$
  converges. Consider such a convergent subsequence $\{ T_+(\phi_{n_m})
  \} \subseteq \{T_+(\phi_n)\}$. By passing to a further subsequence, we can assume that 
  both $\{ T_+(\phi_{n_{m_\ell}})\}$ and $\{ T_+(\psi_{n_{m_\ell}})
  \}$ converge. Let $T_\infty$ denote the limit of the
  latter sequence.  By the above remark, $c_2$ has positive
  translation length on the tree $T_\infty$.

  Let $U \subset \overline{CV}_k$ be a neighborhood of $T_2$.  As the
  set $\{ T_+(\psi_{n_{m_\ell}}) \} \cup \{ T_\infty \}$ is compact,
  and as $c_2$ has positive translation length on every tree therein,
  by \cite[Theorem~13.2]{ar:CL95}, there is an $N$ such that for $\ell
  \geq N$ we have $T_+(\phi_{n_{m_\ell}}) =
  T_+(\psi_{n_{m_\ell}})\delta_2^{-n_{m_\ell}} \in U$. Therefore the subsequence $\{ T_+(\phi_{n_m}) \}$ converges to
  $T_2$.  As this is true for every convergent subsequence of $\{
  T_+(\phi_n) \}$, and as $\overline{CV}_k$ is compact, we have the convergence of 
  $T_+(\phi_n) \to T_2$.  Applying the same argument to $\phi_n^{-1} =
  \delta_2^n\delta_1^{-n}$ we see that $T_-(\phi_n) \to T_1$ as well.
\end{proof}

Fix bases $\mathcal{T}_1$ and $\mathcal{T}_2$ for $F_k$,
obtained from the vertex group(s) (and possibly a choice of
stable letter in the case of an HNN-extension) of the Bass-Serre trees $T_1$ and $T_2$, respectively; to see how this is done, we refer to Section 3.1 of 
\cite{ar:CP}. Let $T_{\mathcal{T}_1}$ and $T_{\mathcal{T}_2}$ be the Cayley trees
for the basis $\mathcal{T}_1$ and $\mathcal{T}_2$, respectively. See
\cite{un:CP2, ar:CP} for the details underlying these constructions.

\begin{proposition}\label{prop:ntwist-T2}
  For sufficiently large $n$, we have:
  \[ \tau_{c_2}(T_{\CT_2},\Lambda_-(\phi_n)) \geq \frac{n}{2}. \]
\end{proposition}

\begin{proof}
  The proposition follows from a slight modification of the arguments from
  Theorem 5.2 in \cite{un:CP2}.  In its proof (equation (5.9)), we
  showed that for for every $\epsilon >0 $ and integer $r > 0$,
  there is an $N > 0$ such that for $n \geq N$:
  \begin{equation}\label{eqn:5.9}
    \frac{ \la c_2^r, \mu_-(\phi_n) \ra_{T_{\CT_2}}   }
    {\omega_{T_{\CT_2}}(\mu_-(\phi_n))} > 1 - \epsilon.
  \end{equation}
  The $\omega_{T_{\CT_2}}(\cdot)$ in the demoninator is a normalization factor whose only relevant value to the present discussion is $\omega_{T_{\CT_2}}(\mu_-(\phi_n))$; it may thus be treated as a positive constant.\footnote{On the other hand, to recognize (\ref{eqn:5.9}) from equation (5.9) in the proof of  \cite[Theorem 5.2]{un:CP2}, it should be observed that 
  $\omega_{T_{\CT_2}}(\eta_{c_2})=\langle c_2^r, \eta_{c_2} \rangle_{T_{\CT_2}}$.}
   Equation (\ref{eqn:5.9}) shows that there is a leaf of
  $\Lambda_-(\phi_n)=Supp(\mu_-(\phi_n))$ contained in the cylinder
  $Cyl_T(c_2^{r})$, and hence $\tau_{c_2}(T_{\CT_2},\Lambda_-(\phi_n)) \geq r$.

Following the same analysis as in \cite[Theorem~5.2]{un:CP2}, fixing $r=n/2$, one can show:
  \begin{equation}\label{eqn:5.9ii}
    \left| 1 - \frac{\la c_2^{n/2}, \mu_-(\phi_n) \ra_{T_{\CT_2}}}
      {\omega_{T_{\CT_2}}(\mu_-(\phi_n))} \right| <
    \left|\frac{\frac{1}{2}An^2 + A_1n + A_2}{An^2 - B_1n - B_2} \right| 
    + \frac{\epsilon}{2}  
  \end{equation}
  for some fixed positive constants $A,A_1,A_2,B_1$ and
  $B_2$.\footnote{Compare this to equation (5.9) from \cite[Theorem
    5.2]{un:CP2} where the numerator of the righthand side is linear
    in $n$, and note that the constants $\beta_1, \beta_2$ there
    depend on $r$. Here in (\ref{eqn:5.9ii}), the numerator is
    quadratic because of the choice of $r = n/2$.}  Thus for large
  enough $n$, we have that $\la c^{n/2},\mu_-(\phi_n) \ra_T > 0$, and
  hence there is a leaf of $\Lambda_-(\phi_n) = Supp(\mu_-(\phi_n))$
  contained in the cylinder $Cyl_T(c_2^{n/2})$. This implies that
  $\tau_{c_2}(T_{\CT_2},\Lambda_-(\phi_n)) \geq n/2$, as claimed.
\end{proof}

We can use the fact that the sequence of trees $\{ T_-(\phi_n) \}$
converges to $T_1$ to show that the twist of $T_-(\phi_n)$ with
$\Lambda_-(\phi_n)$ relative to $c_2$ is also approximately at least
$n$.

\begin{proposition}\label{prop:ntwist}
  There exists a constant $D \geq 1$ such that for sufficiently large
  $n$:
  \[ \tau_{c_2}(T_-(\phi_n),\Lambda_-(\phi_n)) \geq \frac{n}{D}. \]
\end{proposition}

\begin{proof}
  As before, let $T_{\CT_2}$ be the Cayley tree corresponding to the basis $\CT_2$. By Proposition \ref{prop:ntwist-T2}, for each $n$ there is a leaf
  $\ell_n \co \BR \to T_{\CT_2}
  $ of $\Lambda_-(\phi_n)$ 
  that intersects the axis of $c_2$ in a segment of length at least
  $n\ell_T(c_2)/2$.  We must verify that this overlap is not significantly reduced when mapping to $T_-(\phi_n)$.

  Fix an $F_k$--equivariant map $f\co T_{\CT_2} \to T_1$ and scale the metric
  on $T_1$ so that $\Lip(f) \leq 1$, and thus $BBT(f) \leq 1$ (we 
  assume that the volume of $T_{\CT_2}/F_k$ is 1).  By scaling the metrics on
  $T_-(\phi_n)$ we have the convergence of $T_-(\phi_n) \to T_1$ from Theorem \ref{th:treeconvergence}. Thus for large
  enough $n$, we can choose equivariant maps $f_n \co T_{\CT_2} \to T_-(\phi_n)$ so that 
  $\Lip(f_n) \leq 2$ and so $BBT(f_n) \leq 2$.  As convergence is in the
  space of length functions, 
  and as $\ell_{T_1}(c_2) > 0$, there is 
  $\delta > 0$ such that $0 < \delta < \ell_{T_-^n}(c_2) < 1/\delta$
  for all $n$.

  Now let $x_n \in \ell_n(\BR) \cap T_{\CT_2}^{\la c_2 \ra}$ be such that $y_n =
  c_2^{n/2}x_n \in \ell_n(\BR) \cap T_{\CT_2}^{\la c_2 \ra}$.  Thus the path
  $[f_n(x_n), f_n(y_n)]$ contains an arc of the axis of $c_2$ in
  $T_-(\phi_n)$ of length at least $\frac{n}{2}\ell_{T_-(\phi_n)}(c_2)$.  Further
  notice that the distance from either $f_n(x_n)$ or $f_n(y_n)$ to this arc
  is at most 2 (an upper bound for the bounded back tracking constant).

  As $\ell_n \co \BR \to T_{\CT_2}$ is a leaf of the unstable lamination, after tightening
  its image in $T_-(\phi_n)$ we obtain a geodesic $[f_n(\ell_n(\BR))]$, and
  the same statement in the previous paragraph for the segment
  $[f_n(x_n),f_n(y_n)]$ and the axis of $c_2$ holds in turn for $[f_n(x_n),f_n(y_n)]$
  and the geodesic $[f_n(\ell_n(\BR))]$.  Hence the leaf of $\Lambda_-(\phi_n)$ whose image in $T_-(\phi_n)$ is $[f_n(\ell_n(\BR)]$ intersects the axis of $c_2$ along a segment of length at least:
  \begin{align*}
    \frac{n\ell_{T_-(\phi_n)}(c_2)}{2} - 4 &> \frac{n\ell_{T_-(\phi_n)}(c_2)}{2} -
    \frac{4\ell_{T_-(\phi_n)}(c_2)}{\delta}  \\
    &= \ell_{T_-(\phi_n)}(c_2)\left(\frac{n\delta - 8}{4\delta} \right) \\
    & = \ell_{T_-(\phi_n)}(c_2)\frac{n}{D}
  \end{align*}
  for some constant $D > 0$, provided $n > 8/\delta$.  Thus
  \[\tau_{c_2}(T_-(\phi_n),\Lambda_-(\phi_n)) \geq \frac{n}{D}.\]
\end{proof}

It follows from the proposition, together with Theorem \ref{th:alg-twist}, that for each element $\phi_n$, there is a train-track map 
$g_n\co
G_n \to G_n$ such that $\ell_{G_n}(c_2) \leq D'/n$ for some constant $D'$. Note that we are using the fact that every graph on the axis of $\phi$ represents a train-track of $\phi$.

Recall that we assumed that $\{ c_1, c_2 \}$ is not separable in $F_k$. 

\begin{lemma}\label{lm:separable}
  If $\{c_1, c_2 \}$ is not separable, then for large enough $n$,
  neither is $\{c_2,\delta^n_1(c_2)\}$.
\end{lemma}
\begin{proof}
  This is easy to see using Whitehead graphs.  Since the set $\{
  c_1,c_2 \}$ is not separable, the union of their Whitehead graphs is
  connected and does not have a cut vertex (in an appropriate basis)
  \cite{col:Stallings99}.  As cancellation is bounded, for large
  enough $n$, the subword representing $c_1$ will appear as a subword
  of $\delta_1^n(c_2)$.  Hence the union of the Whitehead graphs
  of $c_2$ and $\delta_1^n(c_2)$ will cover the union of the
  Whitehead graphs of $c_1$ and $c_2$.  In particular, their union
  will be connected and will not have a cut vertex.  This implies the
  set $\{ c_2, \delta_1^n(c_2) \}$ is not separable.
\end{proof}

Note that $\phi_n(c_2) = \delta_1^n(c_2)$ so that, as a consequence of the lemma, every edge in the track-track graph $G_n$ must be crossed by either
$c_2$ or $\phi_n(c_2)$.  Therefore, the length of $\phi_n(c_2)$ is
at least $1 - D'/n = (n-D')/n$, and thus the Lipschitz constant for
$\phi_n$ is at least 
$$\frac{(n-D')/n}{D'/n} = n/D' -1.$$
In particular, we have now shown that for some constant $K_1>0$:
\[ \frac{1}{K_1} \log n \leq tr_{CV_k}(\phi_n) \] where
$tr_{CV_k}(\phi) = \min \{ d_L(G,G\phi) \, | \, G \in CV_k \}$ is the
minimal translation length of the element $\phi$.

We obtain the corresponding upper bound on $tr_{CV_k}(\phi_n)$ by
explicitly constructing a path by piecing together geodesic segments
such as those constructed in Example \ref{ex:twist}.

As before, let $T_{\mathcal{T}_1}$ and $T_{\mathcal{T}_2}$ be the
Cayley trees for the basis $\mathcal{T}_1$ and $\mathcal{T}_2$,
respectively.  We consider these trees as points in $CV_k$, with every
edge of each tree having length $1/k$.  We first connect
$T_{\mathcal{T}_2}\delta_2^{n}$ to $T_{\mathcal{T}_2}$ by a geodesic of length $\sim \log n$. Then we
follow an optimal path $P$ from $T_{\mathcal{T}_2}$ to
$T_{\mathcal{T}_1}$, and then connect $T_{\mathcal{T}_1}$ to
$T_{\mathcal{T}_1}\delta_1^{n}$ with a geodesic which has length $\sim
\log n$.  Finally, using the $\delta_1^{n}$--translate $P
\delta_1^{n}$ of $P$, we connect $T_{\mathcal{T}_1}\delta_1^{n}$ to
$T_{\mathcal{T}_2}\delta_1^{n}$ (see Figure \ref{fig:path}).  As the
length of $P$ is independent of $n$, translating the entire path by
$\delta_2^{-n}$, we have for all $n$:
\[ d_L(T_{\mathcal{T}_2},T_{\mathcal{T}_2}\phi_n ) \leq K_2\log n \] for
some $K_2>0$.

\bigskip
\begin{figure}[ht]
  \centering
  \labellist
  \small\hair 2pt
  \pinlabel {$T_{\mathcal{T}_2}\delta_2^n$} at 0 42
  \pinlabel {$T_{\mathcal{T}_2}$} at 210 42
  \pinlabel {$T_{\mathcal{T}_2}\delta_1^{n}$} at 410 42
  \pinlabel {$T_{\mathcal{T}_1}$} at 135 -10
  \pinlabel {$T_{\mathcal{T}_1}\delta_1^{n}$} at 335 -10
  \pinlabel {$\sim \log n$} at 95 35
  \pinlabel {$\sim \log n$} at 230 -3
  \endlabellist
  \includegraphics[scale=0.7]{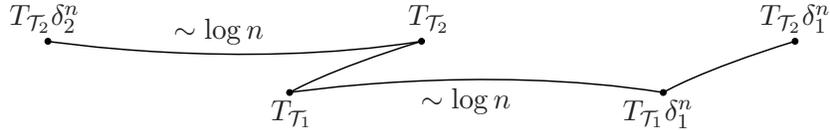}
  \caption{A path from $T_{\mathcal{T}_2}\delta_2^n$ to
    $T_{\mathcal{T}_2}\delta_1^n$.}
  \label{fig:path}
\end{figure}

Combining this upper bound with the previous lower bound, we have
established the following:

\begin{theorem}\label{th:twist-translation}
  Let $T_1,T_2$ be two cyclic trees that fill with associated Dehn
  twist automorphisms $\delta_1$ and $\delta_2$ and let $c_1$, $c_2$
  denote the respective edge stabilizers. Suppose that $\{ c_1,c_2 \}$
  is not separable and that $c_1$ and $c_2$ are not simultaneouly
  elliptic in $\overline{CV}_k$. For $n \geq 1$, let $\phi_n$ be the
  outer automorphism induced by $\delta_1^n\delta_2^{-n}$.  Then there
  is a constant $K = K(T_1,T_2)$ such that for large enough $n$:
  \begin{enumerate}
  \item there is a train-track representative $g_n\co G_n \to G_n$
    such that $\ell_{G_n}(c_2) \leq K/n$, and
  \item $\frac{1}{K} \log n \leq tr_{CV_k}(\phi_n) \leq K \log n$. 
  \end{enumerate}
\end{theorem}

\bibliography{bibliography}

\bibliographystyle{siam}

\end{document}